\documentclass[12pt]{article}
\usepackage{tikz}
\usetikzlibrary{arrows}
\usepackage[framemethod=tikz]{mdframed}
\usepackage{wrapfig}
\usepackage{amsmath,amssymb}
\usepackage{type1cm}
\usepackage{amsthm}
\usepackage{mathrsfs}
\usepackage{mathtools}
\usepackage{enumerate}
\usepackage[all]{xy} 
\usepackage[vcentermath,enableskew]{youngtab}
\usepackage{ytableau}
\usepackage{diagbox}
\AtBeginDocument{%
   \def\MR#1{}
}

\DeclareMathOperator{\GL}{GL}

\DeclareMathOperator{\Irr}{Irr}
\DeclareMathOperator{\pr}{pr}
\DeclareMathOperator{\wt}{wt}
\DeclareMathOperator{\red}{red}
\DeclareMathOperator{\FE}{FE}
\DeclareMathOperator{\op}{op}

\DeclareMathOperator{\supp}{supp}
\DeclareMathOperator{\de}{def}

\newcommand\F{\mathbb{F}}

\newcommand\fp{\mathfrak p}

\newcommand\aFq{\overline{\mathbb F}_q}

\newcommand\cA{\mathcal A}
\newcommand\cB{\mathcal B}
\newcommand\cO{\mathcal O}

\newcommand\cG{\mathcal G}

\newcommand\Gm{\mathbb G_m}

\newcommand\N{\mathbb N}
\newcommand\Q{\mathbb Q}

\newcommand\A{\mathbb A}
\newcommand\B{\mathbb B}

\newcommand\hG{\widehat G}

\newcommand\Z{\mathbb Z}
\newcommand\tW{\widetilde W}

\newcommand\te{\tilde e}
\newcommand\tf{\tilde f}

\newcommand\bl{\bullet}
\newcommand\tp{\mathrm{top}}
\newcommand\ld{\lambda}
\newcommand\up{\upsilon}
\newcommand\ve{\varepsilon}
\newcommand\vp{\varpi}
\newcommand\Y{X_*(T)}
\newcommand\ft{\footnotesize}
\newcommand\bb{\bold b}
\newcommand\mn{\lfloor \frac{m}{n} \rfloor}
\newcommand\la{\langle}
\newcommand\ra{\rangle}

\theoremstyle{definition}
\newtheorem{claim}{Claim}

\theoremstyle{definition}
\newtheorem{theo}{Theorem}[section]
\newtheorem{prop}[theo]{Proposition}
\newtheorem{defi}[theo]{Definition}
\newtheorem{lemm}[theo]{Lemma}
\newtheorem{coro}[theo]{Corollary}
\newtheorem{exam}[theo]{Example}
\newtheorem{rema}[theo]{Remark}

\newtheorem{thm}{Theorem}[section]

\begin{document}
\title{Semi-Modules and Crystal Bases via Affine Deligne-Lusztig Varieties}
\author{Ryosuke Shimada}
\date{}
\maketitle

%\footnotetext[0]{Graduate School of Mathematical Sciences, The University of Tokyo, 3--8--1 Komaba, Meguro-ku, Tokyo, 153--8914, Japan

%E-mail address: \texttt{rshimada@ms.u-tokyo.ac.jp}}

\begin{abstract}
There are two combinatorial ways of parameterizing the $J_b(F)$-orbits of the irreducible components of affine Deligne-Lusztig varieties for $\GL_n$ and superbasic $b$. 
One way is to use the extended semi-modules introduced by Viehmann.
The other way is to use the crystal bases introduced by Kashiwara and Lusztig.
In this paper, we give an explicit correspondence between them using the crystal structure.
\end{abstract}

\section{Introduction}
\label{introduction}
Let $F$ be a non-archimedean local field with finite field $\F_q$ of prime characteristic $p$, and let $L$ be the completion of the maximal unramified extension of $F$.
Let $\sigma$ denote the Frobenius automorphism of $L/F$.
Further, we write $\cO,\ \fp$ for the valuation ring and the maximal ideal of $L$.
Finally, we denote by $\vp$ a uniformizer of $F$ (and $L$) and by $v_L$ the valuation of $L$ such that $v_L(\vp)=1$.

Let $G$ be a split connected reductive group over $F$ and let $T$ be a split maximal torus of it.
Let $B$ be a Borel subgroup of $G$ containing $T$. 
For a cocharacter $\mu\in X_*(T)$, let $\vp^{\mu}$ be the image of $\vp\in \mathbb G_m(F)$ under the homomorphism $\mu\colon\mathbb G_m\rightarrow T$.

Set $K=G(\cO)$.
We fix a dominant cocharacter $\mu\in X_*(T)_+$ and $b\in G(L)$.
Then the affine Deligne-Lusztig variety $X_{\mu}(b)$ is the locally closed reduced $\aFq$-subscheme of the affine Grassmannian $\cG r$ defined as
$$X_{\mu}(b)(\aFq)=\{xK\in G(L)/K\mid x^{-1}b\sigma(x)\in K\vp^{\mu}K\}\subset \cG r(\aFq).$$
Left multiplication by $g^{-1}\in G(L)$ induces an isomorphism between $X_\mu(b)$ and $X_\mu(g^{-1}b\sigma(g))$.
Thus the isomorphism class of the affine Deligne-Lusztig variety only depends on the $\sigma$-conjugacy class of $b$.

The affine Deligne-Lusztig variety $X_{\mu}(b)$ carries a natural action (by left multiplication) by the group
$$J_b(F)=\{g\in G(L)\mid g^{-1}b\sigma(g)=b\}.$$

For $\mu_\bl=(\mu_1, \ldots, \mu_d)\in \Y^d_+$ and $b_\bl=(1,\ldots, 1, b)\in G^d(L)$ with $b\in G(L)$, we can similarly define $X_{\mu_\bl}(b_\bl)\subset \cG r^d$ and $J_{b_\bl}(F)$ using $\sigma_\bl$ given by
$$(g_1,g_2,\ldots, g_d)\mapsto (g_2,\ldots, g_d, \sigma(g_1)).$$

The geometric properties of affine Deligne-Lusztig varieties have been studied by many people.
One of the most interesting results is an explicit description of the set $J_b(F)\backslash \Irr X_\mu(b)$ of $J_b(F)$-orbits of $\Irr X_\mu(b)$, where $\Irr X_\mu(b)$ denotes the set of irreducible components of $X_\mu(b)$.

\begin{rema}
It is known that $X_\mu(b)$ is equi-dimensional.
In the equal characteristic case, this was proved in \cite{HV3}.
In the mixed characteristic case, this was proved in \cite{Takaya}.
\end{rema}

Let $\hG$ be the Langlands dual of $G$ defined over $\overline \Q_l$ with $l\neq p$.
Denote $V_\mu$ the irreducible $\hG$-module of highest weight $\mu$.
The crystal basis $\B_\mu$ of $V_\mu$ was first constructed by Kashiwara and Lusztig (cf.\ \cite{HK}).
In $\Y$, there is a distinguished element $\ld_b$ determined by $b$.
It is the ``best integral approximation'' of the Newton vector of $b$, but we omit the precise definition.
For this, see \cite[\S2.1]{HV2} (in fact, \cite[Example 2.3]{HV2} is enough for our purpose).
In \cite{Nie}, Nie proved that there exists a natural bijection
$$J_b(F)\backslash \Irr X_\mu(b)\cong \B_\mu(\ld_b).$$
In particular, $|J_b(F)\backslash \Irr X_\mu(b)|=\dim V_\mu(\ld_b)$.
The proof is reduced to the case where $G=\GL_n$ and $b$ is superbasic.
So this case is particularly important.
This theorem is first conjectured by Miaofen Chen and Xinwen Zhu.
Before the work by Nie, Xiao-Zhu \cite{XZ} proved the conjecture under the assumption that $b$ is unramified, and Hamacher-Viehmann \cite{HV2} proved the minuscule case.
The last equality is also proved by Rong Zhou and Yihang Zhu in \cite{ZZ}.
See \cite[\S 1.2]{ZZ} for the history.

On the other hand, in the case where $G=\GL_n$ and $b$ is superbasic, Viehmann \cite{Viehmann} defined a stratification of $X_\mu(b)$ using extended semi-modules.
For $\mu\in \Y_+$ and superbasic $b\in \GL_n(L)$, let $\A^\tp_{\mu,b}$ be the set of top extended semi-modules, that is, the extended semi-modules whose corresponding strata are top-dimensional.
Then $J_b(F)\backslash \Irr X_\mu(b)$ is also parametrized by $\A^\tp_{\mu,b}$.

In \cite[Remark 0.10]{Nie}, Nie pointed out that it would be interesting to give an explicit correspondence between $\A^\tp_{\mu,b}$ and $\B_\mu(\ld_b)$. 
The purpose of this paper is to study this question (for the split case).
More precisely, we will propose a way of constructing (the unique lifts of) all the top extended semi-modules from crystal elements, which was unclear before this work.

From now and until the end of this paper, we set $G=\GL_n$.
Let $T$ be the torus of diagonal matrices, and we choose the subgroup of upper triangular matrices $B$ as Borel subgroup. 
Let us define the Iwahori subgroup $I\subset K$ as the inverse image of the {\it lower} triangular matrices under the projection $K\rightarrow G(\aFq),\  \vp\mapsto 0$.

We assume $b$ to be superbasic, i.e., its Newton vector $\nu_b\in \Y_\Q\cong \Q^n$ is of the form $\nu_b=(\frac{m}{n},\ldots, \frac{m}{n})$ with $(m,n)=1$.
Moreover, we choose $b$ to be $\eta^m$, where 
$\eta={\begin{pmatrix}
0 & \vp \\
1_{n-1} & 0\\
\end{pmatrix}}$.
We often regard $\eta$ (and hence $b$) as an element of the Iwahori-Weyl group $\tW$.
For superbasic $b$, the condition that $X_\mu(b)$ (resp.\ $X_{\mu_\bl}(b_\bl)$) is non-empty is equivalent to $v_L(\det(\vp^\mu))=v_L(\det(b))$ (resp.\ $v_L(\det(\vp^{\mu_1+\cdots+\mu_d}))=v_L(\det(b))$) (cf.\ \cite[Theorem 3.1]{He}).
In this paper, we assume this.

Since $X_\mu(b)=X_{\mu+c}(\vp^c b)$ for any central cocharacter $c$,
we may assume that $\mu(1)\geq\cdots\geq \mu(n-1)\geq \mu(n)=0$, where $\mu(i)$ denotes the $i$-th entry of $\mu$.

To state the main result, we introduce $\cA_{\mu_\bl, b_\bl}^{\tp}$ and $\A_{\mu_\bl, b_\bl}^{\tp}$.
See \S\ref{Irreducible Components} for details.
For minuscule $\mu_\bl\in \Y^d_+$ and $b_\bl=(1,\ldots,1,b)\in G^d(L)$,
we define
$$\cA_{\mu_\bl, b_\bl}^{\tp}\coloneqq \{\ld_\bl\in \Y^d\mid \dim X_{\mu_\bl}^{\ld_\bl}(b_\bl)=\dim X_{\mu_\bl}(b_\bl)\}.$$
Here $X_{\mu_\bl}^{\ld_\bl}(b_\bl)$ denotes $X_{\mu_\bl}(b_\bl)\cap It^{\ld_\bl} K/K$.
For $\ld_\bl, \ld'_\bl\in \cA_{\mu_\bl, b_\bl}^{\tp}$, we write $\ld_\bl\sim \ld'_\bl$ if $\ld_\bl=\eta^k \ld'_\bl=(\eta^k \ld'_1,\ldots, \eta^k \ld'_d)$ for some $k\in \Z$.
Let $\A_{\mu_\bl, b_\bl}^{\tp}$ denote the set of equivalence classes with respect to $\sim$, and let $[\ld_\bl]\in \A_{\mu_\bl, b_\bl}^{\tp}$ denote the equivalence class represented by $\ld_\bl\in \cA_{\mu_\bl, b_\bl}^{\tp}$.
Then $J_{b_\bl}(F)\backslash \Irr X_{\mu_\bl}(b_\bl)$ is parametrized by $\A_{\mu_\bl, b_\bl}^{\tp}$.

For $\mu\in\Y_+$, let $\mu_\bl\in \Y_+^d$ be a certain minuscule dominant cocharacter with $\mu=\mu_1+\mu_2+\cdots+\mu_n$, see \S\ref{construction}.
Note that $\{\mu_1,\mu_2,\ldots,\mu_n\}$ itself is uniquely determined by $\mu$.
Let $\pr\colon \cG r^d\rightarrow \cG r$ be the projection to the first factor.
This induces $\pr\colon \A_{\mu_\bl, b_\bl}^{\tp}\rightarrow \sqcup_{\mu'\le \mu}\A_{\mu', b}^\tp$.
Then our main result is the following:
\begin{thm}[Theorem \ref{constructionthm}]
For $\bb\in \B_\mu(\ld_b)$, using the crystal structure of $\B_\mu$, we can construct $\ld_\bl^1(\bb), \ld_\bl^2(\bb),\ldots, \ld_\bl^n(\bb)\in \cA_{\mu_\bl, b_\bl}^{\tp}$ such that $\ld^i_\bl(\bb)=\eta^{i-1}\ld_\bl^1(\bb)$ and $[\ld_\bl^1(\bb)]$ is the unique equivalence class in $\A_{\mu_\bl, b_\bl}^{\tp}$ whose image $\pr([\ld_\bl^1(\bb)])$ belongs to $\A_{\mu, b}^{\tp}$ and maps to $\bb$ under the bijection $J_b(F)\backslash \Irr X_\mu(b)\cong \B_\mu(\ld_b)$ by Nie.
\end{thm}
A crystal is a finite set with a weight map $\wt$ and Kashiwara operators $\te_\alpha$ and $\tf_\alpha$ satisfying certain conditions, see \S\ref{crystal}. 
For more details on the construction of $\ld_\bl^1(\bb), \ld_\bl^2(\bb),\ldots, \ld_\bl^n(\bb)$, see \S\ref{construction}.
The merit of constructing $[\ld_\bl^1(\bb)]$ instead of constructing $\pr([\ld_\bl^1(\bb)])$ directly is that the $J_b(F)$-orbit in $X_\mu(b)$ corresponding $[\ld_\bl^1(\bb)]$ is much more explicit.
It is just $J_b(F)\pr(\overline{X_{\mu_\bl}^{\ld^1_\bl(\bb)}(b_\bl)})$.

In \cite{Shimada4}, the author used Theorem A to find top (non-)cyclic extended semi-modules (see \cite[Definition 3.4]{Viehmann} for the notion of cyclic extended semi-modules).
This is one of the technical cores to study the semi-module stratification and to prove the main theorem there.
In a future work, we shall also explore the possibility of applying Theorem A to determine the type of the stabilizer in $J_b(F)$ of each irreducible component for general $G$ and $b$ (see \cite[Theorem 0.11]{Nie}).

The paper is organized as follows.
In \S\ref{notation}, we fix notation and give an overview of extended semi-modules.
In \S\ref{crystal}, we recall a notion of crystals and a realization by Young tableaux.
In \S\ref{Semi-Modules and Crystal Bases}, we first recall a known result on a relationship between semi-modules and crystal bases for the minuscule case.
After that, we state a precise way of constructing top extended semi-modules from crystal elements.
One of the keys for this construction is to recover Weyl group elements defined in \cite{Nie} (see Remark \ref{constructionthmrema}).
This will be done by applying some Kashiwara operators on $\bb\in \B_\mu(\ld_b)$.
We also need a certain dominant cocharacter defined by a crystal element, which originally comes from \cite{XZ} (see Remark \ref{constructionthmrema2}).
In \S\ref{proof}, we prove the main theorem in a combinatorial way.

\textbf{Acknowledgments:}
The author would like to thank Eva Viehmann, Ulrich G\"ortz and Yihang Zhu for helpful comments.
The author would also like to thank Sian Nie for sharing the possibility of application to determining the type of stabilizers.
The author is grateful to his advisor Yoichi Mieda for his constant support and encouragement.

This work was supported by the WINGS-FMSP program at the Graduate School of Mathematical Science, the University of Tokyo. 
This work was also supported by JSPS KAKENHI Grant number JP21J22427.

\section{Notations}
\label{notation}
Keep the notations and assumptions in \S\ref{introduction}.
\subsection{Basic Notations}
Let $\Phi=\Phi(G,T)$ denote the set of roots of $T$ in $G$.
We denote by $\Phi_+$ (resp.\ $\Phi_-$) the set of positive (resp.\ negative) roots distinguished by $B$.
Let $\chi_{ij}$ be the character $T\rightarrow \Gm$ defined by $\mathrm{diag}(t_1,t_2,\ldots, t_n)\mapsto t_i{t_j}^{-1}$.
Using this notation, we have $\Phi=\{\chi_{i,j}\mid i\neq j\}$, $\Phi_+=\{\chi_{i,j}\mid i< j\}$ and $\Phi_-=\{\chi_{i,j}\mid i> j\}$.
Let $\Delta=\{\chi_{i,i+1}\mid 1\le i <n\}$ be the set of simple roots and $\Delta^\vee$ be the corresponding set of simple coroots.
We let $$\Y_+=\{\mu\in \Y| \la \alpha, \mu \ra\geq0\ \text{for all}\ \alpha\in \Phi_+\}$$ denote the set of dominant cocharacters.
Through the isomorphism $\Y\cong \Z^n$, ${\Y}_+$ can be identified with the set $\{(m_1,\cdots, m_n)\in \Z^n|m_1\geq \cdots \geq m_n\}$.
For $\ld, \mu\in \Y$, we write $\ld\le \mu$ if $\mu-\ld$ is a linear combination of simple coroots with non-negative coefficients.

Let $W_0$ denote the finite Weyl group of $G$, i.e., the symmetric group of degree $n$.
For $1\le i\le n-1$, let $s_i$ be the adjacent transposition changing $i$ to $i+1$.
Then $(W_0, \{s_1,\ldots, s_{n-1}\})$ is a Coxeter system, and we denote by $\ell$ the associated length function.
Let $\le$ denote the Bruhat order on $(W_0, S)$.
For $w\in W_0$, we denote by $\supp(w)$ the set of integers $1\le i\le n-1$ such that the simple reflection $s_i$ appears in some/any reduced expression of $w$.
We say $w\in W_0$ is a Coxeter element (resp.\ partial Coxeter element) if it is a product of simple reflections, and each simple reflection appears exactly once (resp.\ at most once).
Let $\tW$ be the Iwahori-Weyl group of $G$.
Then $\tW$ is isomorphic to
$$\Y\rtimes W_0=\{\vp^\ld w\mid \ld\in \Y, w\in W_0\},$$
and acts on $\Y$.
The action of $\vp^\ld w\in \tW$ is given by $v\mapsto w(v)+\ld$.

\subsection{Extended Semi-Modules}
\label{Extended Semi-Modules}
Here we briefly summarize the definition of extended semi-modules in a combinatorial way,
although we do not need it in this paper.
See \cite{Viehmann} for the precise definition.
Recall that $b\in G(L)$ is a superbasic element with slope $\frac{m}{n}$.
\begin{defi}
A {\it semi-module} for $m, n$ is a subset $A\subset \Z$ that is bounded below and satisfies $m+A\subset A$ and $n+A\subset A$.
Set $\bar{A}=A\setminus (n+A)$.
The semi-module $A$ is called normalized if $\sum_{a\in \bar{A}}a=\frac{n(n-1)}{2}$.
An {\it extended semi-module} $(A,\varphi)$ for $\mu$ is a normalized semi-module $A$ for $m,n$ together with a function $\varphi\colon \Z\rightarrow \N\cup\{-\infty\}$ satisfying certain conditions.
\end{defi}

Set $X_\mu(b)^0=\{xK\in X_\mu(b)\mid v_L(\det(x))=0\}$.
For an extended semi-module $(A,\varphi)$, we can define a locally closed subset $S_{A,\varphi}\subset X_\mu(b)^0$.
They define a decomposition of $X_\mu(b)^0$ into finitely many disjoint locally closed subschemes.
Moreover, $S_{A,\varphi}\subset X_\mu(b)^0$ is irreducible.
So $J_b(F)\backslash \Irr X_\mu(b)$ is parametrized by $\A_{\mu,b}^\tp\coloneqq \{(A,\varphi)\mid \dim S_{A,\varphi}=\dim X_\mu(b)\}$.
In \cite{Viehmann}, extended semi-modules were used to prove the dimension formula (for $X_\mu(b)\neq \emptyset$) 
$$\dim X_\mu(b)=\la \rho, \mu-\nu_b\ra-\frac{1}{2}\de(b).$$
Here $\rho$ denotes half the sum of positive roots, $\nu_b$ denotes the Newton vector of $b$, and $\de(b)$ denotes the defect of $b$.

Let us also make a few remarks on $\cA_{\mu_\bl, b_\bl}^{\tp}$ introduced in \S1.
Set $R_{\mu_\bl, b_\bl}(\ld_\bl)=\{(l, \chi_{i,j})\mid 1\le l\le d,  \la \chi_{i,j}, \ld_l^\natural\ra=-1,(\ld_l)_{\chi_{i,j}}\geq 1\}.$
See \S\ref{Irreducible Components} for the notation.
By \cite[Proposition 2.9]{Nie}, $X_{\mu_\bl}^{\ld_\bl}(b_\bl)\neq \emptyset$ if and only if $\ld_\bl^\natural$ is conjugate to $\mu_\bl$.
Moreover, in this case,
$$\dim X_{\mu_\bl}^{\ld_\bl}(b_\bl)=|R_{\mu_\bl, b_\bl}(\ld_\bl)|.$$
Combining this with the dimension formula for $X_\mu(b)$, we have
$$\cA_{\mu_\bl, b_\bl}^{\tp}=\{\ld_\bl\in \Y^d\mid \ld_\bl^\natural\in W_0 \mu_\bl,|R_{\mu_\bl, b_\bl}(\ld_\bl)|=\la \rho, \mu-\nu_b\ra-\frac{1}{2}\de(b)\}.$$
Thus we can actually define $\cA_{\mu_\bl, b_\bl}^{\tp}$ without using affine Deligne-Lusztig varieties.

If $d=1$, $\cA_{\mu_\bl, b_\bl}^{\tp}$ can be canonically identified with $\cA_{\mu, b}^{\tp}$.
This follows from the fact that if $\mu$ is minuscule, then all extended semi-modules for $\mu$ are cyclic (\cite[COROLLARY 3.7]{Viehmann}).

\section{Crystal Bases}
\label{crystal}
Keep the notations and assumptions above.
\subsection{Crystals and Young Tableaux}
In this subsection, we first recall the definition of $\hG$-crystals from \cite[Definition 3.3.1]{XZ}.
After that, we give a realization of crystals by Young tableaux.
This allows us to treat them in a combinatorial way.
\begin{defi}
\label{crystaldefi}
A (normal) {\it $\hG$-crystal} is a finite set $\B$, equipped with a weight map $\wt\colon \B\rightarrow \Y$, and operators $\te_\alpha, \tf_\alpha\colon \B\rightarrow \B\cup \{0\}$ for each $\alpha\in \Delta$, such that
\begin{enumerate}[(i)]
\item for every $\bb\in \B$, either $\te_\alpha\bb=0$ or $\wt(\te_\alpha \bb)=\wt(\bb)+\alpha^\vee$, and either $\tf_\alpha\bb=0$ or $\wt(\tf_\alpha \bb)=\wt(\bb)-\alpha^\vee$,
\item for all $\bb, \bb'\in \B$ one has $\bb'=\te_\alpha \bb$ if and only if $\bb=\tf_\alpha \bb'$, and
\item if $\ve_\alpha, \phi_\alpha\colon \B\rightarrow \Z,\ \alpha\in \Delta$ are the maps defined by
\begin{align*}
\ve_\alpha(\bb)=\max\{k\mid \te_\alpha^k\bb\neq 0\}\ \  \text{and}\ \ \phi_\alpha(\bb)=\max\{k\mid \tf_\alpha^k\bb\neq 0\},
\end{align*}
then we require $\phi_\alpha(\bb)-\ve_\alpha(\bb)=\la \alpha, \wt(\bb)\ra$.
\end{enumerate}
For $\ld\in \Y$, we denote by $\B(\ld)$ the set of elements with weight $\ld$ for $\hG$, called the {\it weight space} with weight $\ld$ for $\hG$.
Let $\B_1$ and $\B_2$ be the two $\hG$-crystals.
A morphism $\B_1\rightarrow \B_2$ is a map of underlying sets compatible with $\wt,\te_\alpha$ and $\tf_\alpha$.
\end{defi}
In the sequel, we write $\te_i$ and $\tf_i$ (resp.\ $\ve_i$ and $\phi_i$) instead of $\te_{\chi_{i,i+1}}$ and $\tf_{\chi_{i,i+1}}$ (resp.\ $\ve_{\chi_{i,i+1}}$ and $\phi_{\chi_{i,i+1}}$) for simplicity.

\begin{exam}
\label{Box}
Set $\B_\Box=\{\ft \young(1), \young(2),\ldots, \young(n)\}$.
We define $\te_i,\tf_i$ and $\wt$ by
\begin{align*}
 \te_i{\ft\young(k)}=
\begin{cases}
{\ft\young(i)} & (k=i+1)\\
0 & (k\neq i+1),
\end{cases}
 \tf_i{\ft\young(k)}=
\begin{cases}
{\ft\fbox{$i+1$}} & (k=i)\\
0 & (k\neq i),
\end{cases}
\wt({\ft\young(k)})={v_k},
\end{align*}
where ${v_k}=(0,\ldots,0,1,0,\ldots, 0)$ with the nonzero component at position $k$.
It is easy to check that this defines a $\hG$-crystal structure on $\B_\Box$.
\end{exam} 

\begin{exam}
\label{highest}
Let $\B_\mu$ be the crystal basis of the irreducible $\hG$-module of highest weight $\mu\in \Y_+$.
Then $\B_\mu$ is a crystal.
We call $\B_\mu$ a {\it highest weight crystal} of highest weight $\mu$ (cf.\ \cite[Definition 3.3.1 (3)]{XZ}).
There exists a unique element $\bb_\mu\in \B_\mu$ satisfying $\te_\alpha\bb_\mu=0$ for all $\alpha$, $\wt(\bb_\mu)=\mu$, and $\B_\mu$ is generated from $\bb_\mu$ by operators $\tf_\alpha$.
In particular, for $\omega_1=(1,0,\ldots, 0)$, we can easily check that $\B_{\omega_1}$ is a crystal isomorphic to $\B_\Box$ and $\bb_{\omega_1}$ corresponds to $\young(1)$.
\end{exam}

Following \cite[Definition 3.3.1(5)]{XZ}, we define the tensor product of $\hG$-crystals.
\begin{defi}
\label{tensor}
Let $\B_1$ and $\B_2$ be two $\hG$-crystals.
The tensor product $\B_1\otimes \B_2$ is the $\hG$-crystal with underlying set $\B_1\times \B_2$, and $\wt(\bb_1\otimes \bb_2)=\wt(\bb_1)+\wt(\bb_2)$.
The operators $\te_\alpha$ and $\tf_\alpha$ are defined by
\begin{align*}
\te_\alpha(\bb_1\otimes \bb_2)&=
\begin{cases}
\te_\alpha \bb_1\otimes \bb_2 & (\phi_\alpha(\bb_1)\geq \ve_\alpha(\bb_2))\\
\bb_1\otimes \te_\alpha \bb_2 & (\phi_\alpha(\bb_1)< \ve_\alpha(\bb_2)),
\end{cases}\\
\tf_\alpha(\bb_1\otimes \bb_2)&=
\begin{cases}
\tf_\alpha \bb_1\otimes \bb_2 & (\phi_\alpha(\bb_1)> \ve_\alpha(\bb_2))\\
\bb_1\otimes \tf_\alpha \bb_2 & (\phi_\alpha(\bb_1)\le \ve_\alpha(\bb_2)).
\end{cases}
\end{align*}
We have
\begin{align*}
\ve_\alpha(\bb_1\otimes \bb_2)&=\max\{\ve_\alpha(\bb_1), \ve_\alpha(\bb_2)-\la \alpha, \wt(\bb_1)\ra\}, \\
\phi_\alpha(\bb_1\otimes \bb_2)&=\max\{\phi_\alpha(\bb_2), \phi_\alpha(\bb_1)+\la \alpha, \wt(\bb_1)\ra\}.
\end{align*}
\end{defi}

Taking tensor product of $\hG$-crystal is associative, making the category of $\hG$-crystals a monoidal category.
Using this fact, we will endow a $\hG$-crystal structure on the set of Young tableaux.
A detailed discussion can be found in \cite{HK}, chapter 7.

\begin{defi}
A {\it Young diagram} is a collection of boxes arranged in left-justified rows with a weakly decreasing number of boxes in each row.
A {\it tableau} is a Young diagram filled with numbers, one for each box.
A {\it semistandard tableau} is a tableau obtained from a Young diagram by filling the boxes with the numbers $1,2,\ldots, n$ subject to the conditions
\begin{enumerate}[(i)]
\item the entries in each row are weakly increasing from left to right,
\item the entries in each column are strictly increasing from top to bottom.
\end{enumerate}
$$\yng(4,3,1)\hspace{2cm}\young(1124,233,4)$$
\end{defi}
We denote by $\cB(Y)$ the set of all semistandard tableaux of shape $Y$.

\begin{defi}
\label{Far-Eastern}
Let $Y$ be a Young diagram and let $N$ be the number of boxes in $Y$.
The {\it Far-Eastern reading} is an embedding $\cB(Y)\rightarrow \B_\Box^{\otimes N}$ defined by decomposing a semistandard tableau $\bb\in \cB(Y)$ into a tensor product of its boxes by proceeding down columns from top to bottom and from right to left.

$$\young(1124,233,4)=\young(4)\otimes \young(2)\otimes \young(3)\otimes \young(1)\otimes \young(3)\otimes \young(1)\otimes \young(2)\otimes \young(4)$$
\end{defi}

\begin{theo}
\label{Young crystal}
Let $Y$ be a Young diagram.
Then the Far-Eastern reading $\cB(Y)\rightarrow \B_\Box^{\otimes N}$ is stable under $\te_i$ and $\tf_i$ for any $i$.
Hence the Far-Eastern reading defines a $\hG$-crystal structure on $\cB(Y)$.
\end{theo}
\begin{proof}
This follows from \cite[Theorem 7.3.6]{HK}.
\end{proof}

For a semistandard tableau $\bb\in \cB(Y)$, let $k_i$ denote the number of $i$'s appearing in $\bb$.
Then the weight map $\wt$ on this $\hG$-crystal structure is given by $\wt(\bb)=(k_1,\ldots, k_n)$.
Finally, the following theorem gives a realization of $\B_\mu$.
\begin{theo} 
\label{Young theorem}
Let $\mu=(\mu(1),\ldots,\mu(n))\in \Y_+\setminus\{0\}$ with $\mu(n)\geq0$.
Let $Y$ be the Young diagram having $\mu(i)$ boxes in the $i$th row.
Then $\B_\mu$ is isomorphic to $\cB(Y)$.
\end{theo}
\begin{proof}
This is \cite[Theorem 7.4.1]{HK}.
\end{proof}
In the sequel, we identify $\B_\mu$ and $\cB(Y)$ by this isomorphism.
The following result is an explicit description of the actions of $\te_i$ and $\tf_i$ on $\B_\mu$.
\begin{theo}
\label{ef}
The actions of $\te_i$ and $\tf_i$ on $\bb\in \B_\mu$ can be computed by following the steps below:
\begin{enumerate}[(i)]
\item In the Far-Eastern reading $\bb_1\otimes\cdots\otimes \bb_N$ of $\bb$,
we identify \framebox[1.2em]{\rule{0pt}{1.6ex}$i$} (resp.\ \fbox{$i+1$}) by $+$ (resp.\ $-$) and neglect other boxes. 
\item Let $u_i(\bb)=u^1u^2\cdots u^\ell\ (u^j\in \{\pm\})$ be the sequence obtained by (i).
If there is ``$+-$'' in $u(\bb)$, then we neglect such a pair.
We continue this procedure as far as we can.
\item Let $u_i(\bb)_{\red}=-\cdots-+\cdots+$ be the sequence obtained by (ii).
Then $\te_i$ changes the rightmost $-$ in $u(\bb)_{\red}$ to $+$, and $\tf_i$ changes the leftmost $+$ in $u(\bb)_{\red}$ to $-$.
If there is no such $-$ (resp.\ $+$), then $\te_i\bb=0$ (resp.\ $\tf_i \bb=0$).
\end{enumerate}
Moreover, $\ve_i(\bb)$ (resp.\ $\phi_i(\bb)$) is equal to the number of $-$ (resp.\ $+$) in $u(\bb)_{\red}$.
\end{theo}
\begin{proof}
The first statement is \cite[Theorem 3.4.2]{KN}.
The second statement follows immediately from this.
\end{proof}

We will see an example of this computation in \S\ref{example}.
For $j_1\le j_2$, let $u^{j_1}u^{j_1+1}\cdots u^{j_2}$ be a part of $u_i(\bb)$ above.
Then similarly as the notation above, we denote by $(u^{j_1}u^{j_1+1}\cdots u^{j_2})_{\red}$ the sequence obtained by neglecting ``$+-$'' as far as we can.
Then Theorem \ref{ef} tells us that $\ve_i(\bb)=\max\{\text{the number of $-$ in $(u^1u^2\cdots u^j)_{\red}$}\mid 0\le j\le \ell\}$ (resp.\ $\phi_i(\bb)=\max\{\text{the number of $+$ in $(u^ju^{j+1}\cdots u^\ell)_{\red}$}\mid 1\le j\le \ell+1\}$).
If $\ve_i(\bb)>0$ (resp.\ $\phi_i(\bb)>0$), then $\te_i$ (resp.\ $\tf_i$) changes $u^j=-$ (resp.\ $u^j=+$) with $j$ minimal (resp.\ maximal) such that the number of $-$ (resp.\ $+$) in $(u^1\cdots u^j)_{\red}$ (resp.\ $(u^j\cdots u^\ell)_{\red}$) is $\ve_i(\bb)$ (resp.\ $\phi_i(\bb)$).

Finally, we recall the Weyl group action on crystals.
Let $\B$ be a $\hG$-crystal.
For any $1\le i\le n-1$ and $\bb\in \B$, we set
\begin{align*}
s_i\bb=
\begin{cases}
\tf_i^{\la \chi_{i, i+1}, \wt(\bb)\ra}\bb & \text{if $\la \chi_{i, i+1}, \wt(\bb)\ra \geq 0$}\\
\te_i^{-\la \chi_{i, i+1}, \wt(\bb)\ra}\bb & \text{if $\la \chi_{i, i+1}, \wt(\bb)\ra \le 0$}.
\end{cases}
\end{align*}
Then we have the obvious relation
$$\wt(s_i\bb)=s_i(\wt(\bb)).$$
By \cite[Theorem 7.2.2]{Kashiwara}, this extends to the action of the Weyl group $W_0$ on $\B$, which is compatible with the action on $\Y$.

\begin{lemm}
\label{Weyl action}
Let $w, w'\in W_0$ and $\bb\in \B$.
If $w(\wt(\bb))=w'(\wt(\bb))$, then $w\bb=w'\bb$.
\end{lemm}
\begin{proof}
It is enough to show that if $w(\wt(\bb))=\wt(\bb)$, then $w\bb=\bb$.
By decomposing $w$ into disjoint cycles and considering the conjugation, we can reduce the general case to the case where $w=s_i$.
Then the assertion follows immediately from the definition of the Weyl group action on crystals.
\end{proof}
Let $\bb\in \B(\ld)$.
If $\ld'$ is a conjugate of $\ld$, i.e., there exists $w\in W_0$ such that $\ld'=w\ld$, then we call $w\bb$ the conjugate of $\bb$ with weight $\ld'$.
By Lemma \ref{Weyl action}, this does not depend on the choice of $w$.

\subsection{The Minuscule Case}
\label{Crystal:The Minuscule Case}
If $\mu\in \Y_+$ is minuscule, then $\wt\colon\B_\mu\rightarrow \Y$ gives an identification between $\B_\mu$ and the set of cocharacters which are conjugate to $\mu$.
Suppose $\mu_\bl=(\mu_1,\ldots,\mu_d)\in \Y_+^d$ is minuscule.
We can also identify $\B_{\mu_\bl}^{\hG^d}:=\B_{\mu_1}\times\cdots\times \B_{\mu_d}$ with the set of cocharacters in $\Y^d$ which are conjugate to $\mu_\bl$.
Under this identification, set $$\B_{\mu_\bl}^{\hG^d}(\ld)=\{(\mu_1',\ldots ,\mu_d')\in\B_{\mu_\bl}^{\hG^d}\mid \mu_1'+\cdots+\mu_d'=\ld\}$$
for any $\ld\in\Y$.

We write $\B_{\mu_\bl}^{\hG}$ for the $\hG$-crystal $\B_{\mu_1}\otimes\cdots\otimes \B_{\mu_d}$.
Note that this is equal to $\B_{\mu_\bl}^{\hG^d}$ as a set.
As a $\hG$-crystal, we can decompose $\B_{\mu_\bl}^{\hG}$ into simple objects, i.e.,
$$\B_{\mu_\bl}^{\hG}=\sqcup_\mu \B_\mu^{m_{\mu_\bl}^\mu}.$$
Here $m_{\mu_\bl}^\mu$ denotes the multiplicity with which $\B_\mu$ appears in $\B_{\mu_\bl}^{\hG}$.
Using this decomposition, we define a natural map
$$\otimes\colon \B_{\mu_\bl}^{\hG^d}\rightarrow \B_{\mu_\bl}^{\hG}\rightarrow \sqcup_\mu \B_\mu$$
as a composition of the map given by taking tensor product and the canonical projection to highest weight $\hG$-crystals.

For $1\le k<n$, let $\omega_k$ be the cocharacter of the form $(1,\ldots,1,0,\ldots,0)$ in which $1$ is repeated $k$ times.
Assume that each $\mu_i$ is equal to $\omega_{k_i}$ for some $1\le k_i<n$ and $i<j$ if and only if $k_i\le k_j$.
In the rest of paper, we call such $\mu_\bl$ {\it Far-Eastern}.
Since $\mu_\bl$ is Far-Eastern, then $|\mu_\bl|:=\mu_1+\cdots+\mu_d$ is dominant and its last entry is $0$.
Set $\mu=|\mu_\bl|$ for some Far-Eastern $\mu_\bl$.
Using Theorem \ref{Young theorem}, we obtain an embedding (i.e., an injective morphism of crystals)
$$\FE\colon \B_\mu\rightarrow \B_{\mu_\bl}^{\hG},$$
which decomposes $\bb\in \B_\mu$ into the tensor product of its columns from right to left.
We also call $\FE$ the Far-Eastern reading.
By forgetting the $\hG$-crystal structure, we obtain a map $\B_\mu\rightarrow \B_{\mu_\bl}^{\hG^d}$, which is also denoted by $\FE$.
$$\young(1124,233,4)=\young(4)\otimes \young(2,3)\otimes \young(1,3)\otimes \young(1,2,4)$$

\begin{lemm}
\label{FE}
For any $\bb\in \B_\mu$, $\FE(\bb)$ is the unique element in $\B_{\mu_\bl}^{\hG^d}$ such that $\otimes(\FE(\bb))=\bb$.
\end{lemm}
\begin{proof}
Let $\bb_\mu\in \B_\mu$ be the unique element with highest weight $\mu$.
Then the $i$th row of $\bb_\mu$ consists of only $i$.
By the ``Littlewood-Richardson" rule (see \cite[Theorem 7.4.6]{HK}), we can check that $m_{\mu_\bl}^\mu=1$ and $\FE(\bb_\mu)\in \B_{\mu_\bl}^{\hG}$ is the unique maximal vector with weight $\mu$.
In particular, $\otimes(\FE(\bb_\mu))=\bb_\mu$.
Since $\FE$ is a morphism of crystals, we have $\FE(\tf_\alpha \bb)=\tf_\alpha \FE(\bb)$ for any $\alpha\in \Delta, \bb\in \B_\mu$.
Therefore $\otimes(\FE(\bb))=\bb$, and such $\bb$ is unique.
\end{proof}

\section{Semi-Modules and Crystal Bases}
\label{Semi-Modules and Crystal Bases}
Keep the notations and assumptions above.
From now, we set $\tau=s_1s_2\cdots s_{n-1}$.
\subsection{Irreducible Components}
\label{Irreducible Components}
Let $\ld\in \Y$ and $\alpha\in \Phi$.
We set $\ld_\alpha=\la \alpha, \ld\ra$ if $\alpha\in \Phi_-$ and $\ld_\alpha=\la \alpha, \ld\ra-1$ if $\alpha\in \Phi_+$.
Let $U_\ld$ be the subgroup of $G$ generated by $U_\alpha$ such that $\ld_\alpha\geq 0$.
We define $\up_\ld\in W_0$ to be the unique element such that $U_\ld=\up_\ld U\up_\ld^{-1}$.
In particular, $\up_{\ld}^{-1}\ld$ is dominant.
Here $U$ denotes the unipotent radical of $B$.
It is easy to check $\up_{\eta \ld}=\tau \up_{\ld}$.
For $\ld_\bl=(\ld_1,\ldots, \ld_d)\in \Y^d$, set $\up_{\ld_\bl}=(\up_{\ld_1},\ldots, \up_{\ld_d})$.

Let us denote by $\Irr X_{\mu_\bl}(b_\bl)$ the set of irreducible components of $X_{\mu_\bl}(b_\bl)$.
Through the identification $J_b(F)\cong J_{b_\bl}(F)$ given by $g\mapsto (g,\ldots,g)$, this set is equipped with an action of $J_b(F)$.
Set $J_b(F)^0=J_b(F)\cap K=J_b(F)\cap I$.
Then we have $J_b(F)/J_b(F)^0=\{\eta^k J_b(F)^0\mid k\in \Z\}$ (cf.\ \cite[Lemma 3.3]{CV}).

We first consider the case where $\mu_\bl$ is minuscule.
For $\ld_\bl\in \Y^d$, set $\ld_\bl^\dag=b_\bl \sigma_\bl(\ld_\bl)$, $\ld_\bl^\natural=\ld_\bl^\dag-\ld_\bl$ and $\ld_\bl^\flat =\up_{\ld_\bl}^{-1}(\ld_\bl^\natural)$.
It is easy to check $(\eta\ld_\bl)^\flat=\ld_\bl^\flat$.
Let $\ld_b$ denote the cocharacter whose $i$-th entry is $\lfloor \frac{im}{n}\rfloor-\lfloor \frac{(i-1)m}{n}\rfloor$.
\begin{theo}
\label{minuscule bijection}
Assume that $\mu_\bl\in \Y_+^d$ is minuscule.
Then $\ld_\bl\in \cA_{\mu_\bl, b_\bl}^\tp$ if and only if $\ld_\bl^\flat\in \B_{\mu_\bl}^{\hG^d}(\ld_b)$, and $X_{\mu_\bl}^{\ld_\bl}(b_\bl)$ is an affine space for such $\ld_\bl$.
Moreover, the maps $\ld_\bl\mapsto \ld_\bl^\flat$ and $\ld_\bl\mapsto \overline{X_{\mu_\bl}^{\ld_\bl}(b_\bl)}$ induce bijections
$$J_{b}(F)\backslash \Irr X_{\mu_\bl}(b_\bl)\cong \A_{\mu_\bl,b_\bl}^\tp\cong \B_{\mu_\bl}^{\hG^d}(\ld_b).$$
\end{theo}
\begin{proof}
This follows from \cite[Proposition 2.9 \& Theorem 3.3]{Nie}.
Note that we have $\mathrm{Stab}_{J_b(F)} (X_{\mu_\bl}^{\ld_\bl}(b_\bl))=J_b(F)^0$.
\end{proof}

We write $\gamma^{G^d}\colon \Irr X_{\mu_\bl}(b_\bl)\rightarrow \B_{\mu_\bl}^{\hG^d}$ for the map which factors through this bijection.
Set $\mu=|\mu_\bl|$.
By \cite[Corollary 1.6]{Nie}, the projection $\pr\colon\cG r^d\rightarrow \cG r$ to the first factor induces a $J_b(F)$-equivariant map
$$\Irr X_{\mu_\bl}(b_\bl)\rightarrow \sqcup_{\mu'\le\mu}\Irr X_{\mu'}(b),\quad C\mapsto \pr(C),$$
which is also denoted by $\pr$.
The general case can be characterized by the minuscule case using $\pr$ and the tensor product of $\hG$-crystals:

\begin{theo}
\label{general bijection}
There exists a map
$$\gamma^G\colon \Irr X_{\mu}(b)\rightarrow \B_{\mu}(\ld_b)$$
which is characterized by the Cartesian square
\[
  \xymatrix@C=30pt{
    \Irr X_{\mu_\bl}(b_\bl) \ar[r]^-{\gamma^{G^d}} \ar[d]_{\pr} & \B_{\mu_\bl}^{\hG^d} \ar[d]^{\otimes} \\
    \sqcup_{\mu'\le\mu} \Irr X_{\mu'}(b) \ar[r]^-{\gamma^G} & \sqcup_{\mu'\le\mu} \B_{\mu'}^{\hG}, 
  }
\]
where $\mu_\bl$ is a minuscule cocharacter in $\Y_+^d$ such that $\mu=|\mu_\bl|$.
Moreover, $\gamma^G$ factors through a bijection
$$J_b(F)\backslash \Irr X_\mu(b)\cong \B_\mu(\ld_b).$$
\end{theo}
\begin{proof}
This follows from \cite[Theorem 0.5 \& Theorem 0.7]{Nie}.
\end{proof}

Let us denote by $\Gamma^{G^d}$ (resp.\ $\Gamma^G$) the bijection $\A_{\mu_\bl,b_\bl}^\tp\rightarrow \B_{\mu_\bl}^{\hG^d}(\ld_b)$ (resp.\ $\A_{\mu,b}^\tp\rightarrow \B_{\mu}(\ld_b)$) induced by $\gamma^{G^d}$ (resp.\ $\gamma^G$).
Then by Theorem \ref{minuscule bijection} and Theorem \ref{general bijection}, we have the Cartesian square
\[
  \xymatrix@C=30pt{
    \A_{\mu_\bl,b_\bl}^\tp \ar[r]^-{\Gamma^{G^d}} \ar[d]_{\pr} & \B_{\mu_\bl}^{\hG^d}(\ld_b) \ar[d]^{\otimes} \\
    \sqcup_{\mu'\le\mu} \A_{\mu',b}^\tp \ar[r]^-{\Gamma^G} & \sqcup_{\mu'\le\mu} \B_{\mu'}^{\hG}(\ld_b), 
  }
\]
where $\mu_\bl$ is a minuscule cocharacter in $\Y_+^d$ such that $\mu=|\mu_\bl|$.

\subsection{Construction}
\label{construction}
Let $\mu\in \Y_+$.
For $1\le k\le \mu(1)$, set
$$
\mu_k=
\begin{cases}
\omega_1 & (1\le k\le \mu(1)-\mu(2)),\\
\omega_2 & (\mu(1)-\mu(2)<k\le \mu(1)-\mu(3)),\\
\hspace{0.2em}\vdots & \\
\omega_{n-2} & (\mu(1)-\mu(n-2)<k\le \mu(1)-\mu(n-1)),\\
\omega_{n-1} & (\mu(1)-\mu(n-1)<k\le \mu(1)).
\end{cases}
$$
Set $d=\mu(1)$.
Obviously $\mu_\bl\in \Y^d_+$ is Far-Eastern (\S\ref{Crystal:The Minuscule Case}) and $\mu=|\mu_\bl|$.

Let $w_{\max}$ denote the maximal length element in $W_0$.
Set $\ld_b^{\op}=w_{\max}\ld_b$.
For any $\bb\in \B_\mu(\ld_b)$, we denote by $\bb^{\op}$ the conjugate of $\bb$ with weight $\ld_b^{\op}$.
Let $1\le m_0<n$ be the residue of $m$ modulo $n$.
Since $\lfloor \frac{im}{n}\rfloor=i\frac{m-m_0}{n}+\lfloor \frac{im_0}{n}\rfloor$, we have $\ld(i)=\mn+\lfloor \frac{im_0}{n}\rfloor-\lfloor \frac{(i-1)m_0}{n}\rfloor$.
So each entry of $\ld_b$ is $\mn$ or $\mn+1$, and $\ld_b(i)=\ld_b(n+1-i)$ for any $2\le i\le n-1$.
For $0\le k\le m_0$, let $1\le i_k\le n$ be the minimal integer such that $\lfloor \frac{i_km_0}{n}\rfloor\geq k$.
In other words, we define $i_0=1<i_1<i_2<\cdots<i_{m_0}=n$ as the integers such that $\ld_b(i_1)=\ld_b(i_2)=\cdots =\ld_b(i_{m_0})=\mn+1$.
Then 
$$\text{$\ld_b^{\op}=w_{\max}'\ld_b$,\quad where $w_{\max}'=(s_{i_{m_0-1}}\cdots s_{n-1})\cdots (s_{i_1}\cdots s_{i_2-1})(s_1\cdots s_{i_1-1})$.}$$
Here $\ld_b(i)=\mn$ (resp.\ $\ld_b(i+1)=\mn$) if and only if $s_{i-1}s_i\le w_{\max}'$ (resp.\ $s_is_{i+1}\le w_{\max}'$).
By Lemma \ref{Weyl action}, it follows that $\bb^{\op}$ can be computed by the action of the Coxeter element $w_{\max}'$.
In this computation, each $s_i$ acts as the action of $\te_i$ because $\mn-(\mn+1)=-1$.
Therefore, if we write $$\FE(\bb)=\bb_1\otimes\cdots \otimes \bb_d,$$
then there exists $(w_1, \ldots, w_{d})\in W_0^d$ such that
$$\FE(\bb^{\op})=w_1\bb_1\otimes \cdots \otimes w_d\bb_d$$
and each simple reflection appears exactly once in some $\supp(w_j)$.
\begin{lemm}
\label{w}
The tuple $(w_1,\ldots,w_d)\in W_0^d$ as above is uniquely determined by $\bb$.
In particular, $w(\bb)\coloneqq w_1^{-1}\cdots w_d^{-1}$ is a Coxeter element uniquely determined by $\bb$.
\end{lemm}
\begin{proof}
If $(w'_1,\ldots,w'_d)\in W_0^d$ is another tuple such that
$$\FE(\bb^{\op})=w'_1\bb_1\otimes \cdots \otimes w'_d\bb_d$$
and each simple reflection appears exactly once in some $\supp(w'_j)$, then each $s_i$ appearing in this tuple acts as the action of $\te_i$.
This follows from the fact that
$$\ld_b^{\op}-\ld_b=(1,0,\ldots,0,-1)=\chi_{1,2}^{\vee}+\chi_{2,3}^{\vee}+\cdots+\chi_{n-1,n}^{\vee}$$
and $\chi_{1,2}^{\vee},\chi_{2,3}^{\vee},\ldots,\chi_{n-1,n}^{\vee}$ are linearly independent. 
Assume that $s_i\in \supp(w_j)$.
If $s_i\notin \supp(w'_j)$, then the number of $1$ appearing at position $k\le i$ of $\wt(w'_j\bb_j)$ is different from that of $\wt(w_j\bb_j)$, which is a contradiction.
So $s_i\in \supp(w'_j)$.
Since this is true for any $i$, it follows that $\supp(w_j)=\supp(w'_j)$ for any $j$.

Fix $j$ and let $\Sigma$ be a connected component of $\supp(w_j)=\supp(w'_j)$.
In particular, $\Sigma=\{\min \Sigma,\min \Sigma+1,\ldots, \max \Sigma-1, \max \Sigma\}$.
We define $k_0=\min \Sigma\le k_1< k_2< \cdots < k_l=\max \Sigma$ by
$$
\wt(\bb_j)(k)=
\begin{cases}
1 & (k=k_1+1, k_2+1,\ldots, k_l+1)\\
0 & (k\neq k_1+1, k_2+1,\ldots, k_l+1)
\end{cases}
$$
for $k_0\le k\le k_l+1$.
Since each $s_i$ with $i\in \supp(w_d)$ acts as the action of $\te_i$, we have
$$(s_{k_{l-1}+1}s_{k_{l-1}+2}\cdots s_{k_l})\cdots (s_{k_1+1}s_{k_1+2}\cdots s_{k_2})(s_{k_0}s_{k_0+1}\cdots s_{k_1})\le w_j.$$
By the above argument, the same is true for $w'_j$.
Since both $j$ and $\Sigma$ are arbitrary, it follows that $w_j=w'_j$.
\end{proof}
We call $w(\bb)$ the {\it Coxeter element associated to $\bb$}.
Set $\Upsilon(\bb)=\{\up\in W_0\mid \up^{-1}\tau^m \up=w(\bb)\}$.
Clearly $|\Upsilon(\bb)|=n$.

For any $\bb'\in \B_{\mu}$, set
$$\xi(\bb')=(\ve_1(\bb')+\cdots+\ve_{n-1}(\bb'),\ve_2(\bb')+\cdots+\ve_{n-1}(\bb'),\ldots,\ve_{n-1}(\bb'),0).$$
Let $\ld_b^-$ be the anti-dominant conjugate of $\ld_b$, and let $\bb^-$ be the conjugate of $\bb$ with weight $\ld_b^-$.
For any $\bb\in \B_{\mu}(\ld_b)$ and $\up\in \Upsilon(\bb)$, we define $\xi_{\bl}(\bb, \up)\in \Y^d$ by
$$\xi_{j}(\bb, \up)=\up \xi(\up^{-1}\bb^-)+\sum_{1\le j'< j}\up w_1^{-1}\cdots w_{j'-1}^{-1}\wt(\bb_{j'})\quad (1\le j\le d).$$

\begin{theo}
\label{constructionthm}
We have $\up_{\xi_j(\bb,\up)}=\up w_1^{-1}\cdots w_{j-1}^{-1}$ and $\xi_{\bl}(\bb,\up)\in \cA_{\mu_\bl, b_\bl}^\tp$.
Moreover, if $\up'$ is an element in $\Upsilon(\bb)$ different from $\up$, then $\xi_{\bl}(\bb, \up)\neq \xi_{\bl}(\bb, \up')$ and $\xi_{\bl}(\bb, \up)\sim \xi_{\bl}(\bb, \up')$.
Finally, we have $$(\Gamma^{G^d})^{-1}(\FE(\bb))=[\xi_{\bl}(\bb,\up)].$$
\end{theo}

Clearly, this construction itself does not depend on the choice of realization of $\B_\mu$.
Note that each entry of $\xi_1(\bb,\up)=\up\xi(\up^{-1}\bb^-)$ is non-negative, and at least one entry is equal to $0$.
So if $\xi_{\bl}(\bb, \up)=\eta^k\xi_{\bl}(\bb, \up')$, then $-n<k<n$.
Then Theorem A follows immediately from Theorem \ref{constructionthm} and this observation.
\begin{rema}
\label{constructionthmrema2}
For $\bb\in \B_\mu(\ld_b)$, $\xi(\bb)$ already appeared in \cite[Lemma 4.4.3]{XZ}.
In \cite[Theorem 4.4.5]{XZ}, $\xi(\bb)$ was used to construct the irreducible component corresponding to $\bb$.
\end{rema}

\begin{rema}
\label{constructionthmrema}
Let $\bb\in \B_\mu(\ld_b)$ and $\up\in \Upsilon(\bb)$.
In \cite[\S3.3]{Nie}, Nie defined $(w'_1,\ldots, w'_d)\in W_0^d$ from $\ld_\bl\coloneqq \xi_{\bl}(\bb, \up)$ as follows.
Set $a_{j,i}=\up_{\ld_j}(i)+n\ld_j(\up_{\ld_j}(i))$ for $1\le j \le d$.
By the definition of $\up_{\ld_j}$, $a_{j,1}>\cdots>a_{j,n}$ is the arrangement of the integers $i+n\ld_j(i)$ in the decreasing order.
Define $(w'_1,\ldots, w'_d)\in W_0^d$ such that
$$a_{j,i}=
\begin{cases}
a_{j+1, w'_j(i)}-n\ld_j^\flat(i)& (1\le j\le d-1)\\
a_{1,w'_d(i)}-n\ld_d^\flat(i)+m& (j=d).
\end{cases}$$
Then we have $(w_1,\ldots, w_d)=(w'_1,\ldots, w'_d)$.
Indeed, by Theorem \ref{constructionthm} and \cite[Lemma 3.7]{Nie}, we have $\up(w_{j-1}\cdots w_1)^{-1}=\up_{\ld_j}=\up (w'_{j-1}\cdots w'_1)^{-1}$ for $1\le j\le d$.
This implies $(w_1,\ldots, w_{d-1})=(w'_1,\ldots, w'_{d-1})$.
Moreover, by \cite[Lemma 3.11]{Nie}, $\ell(w'_d\cdots w'_1)=\sum_{j=1}^d\ell(w'_j)=n-1$ and $w'_d\cdots w'_1$ is a product of distinct simple reflections.
So we have $\supp(w_d)=\supp(w'_d)$.
Let $\Sigma$ be a connected component of $\supp(w_d)=\supp(w'_d)$.
We define $k_0=\min \Sigma\le k_1< k_2< \cdots < k_l=\max \Sigma$ such that
$$(s_{k_{l-1}+1}s_{k_{l-1}+2}\cdots s_{k_l})\cdots (s_{k_1+1}s_{k_1+2}\cdots s_{k_2})(s_{k_0}s_{k_0+1}\cdots s_{k_1})\le w'_d.$$
In particular, $\{\chi_{i,i+1}\in \Delta\mid k_0\le i\le k_l, w'_d\chi_{i,i+1}\in \Phi_-\}=\{\chi_{k_1,k_1+1},\ldots,\chi_{k_l,k_l+1}\}$.
By Theorem \ref{minuscule bijection} and \cite[Lemma 3.8 (1) \& Lemma 3.9]{Nie}, $w'_d\chi_{i,i+1}\in \Phi_-$ if and only if $\wt(\bb_d)(i+1)-\wt(\bb_d)(i)=1$ for $i\in \supp(w_d)$.
Thus we have
$$(\wt(\bb_d)(k_1),\wt(\bb_d)(k_1+1))=\cdots=(\wt(\bb_d)(k_l),\wt(\bb_d)(k_l+1))=(0,1)$$
and $\wt(\bb_d)(k)\geq \wt(\bb_d)(k+1)$ for $k\in \{k_0,k_0+1,\ldots, k_l\}\setminus\{k_1, k_2,\ldots, k_l\}$.
Since each $s_i$ with $i\in \supp(w_d)$ acts as the action of $\te_i$, we have
$$
\wt(\bb_d)(k)=
\begin{cases}
1 & (k=k_1+1, k_2+1,\ldots, k_l+1)\\
0 & (k\neq k_1+1, k_2+1,\ldots, k_l+1)
\end{cases}
$$
for $k_0\le k\le k_l+1$, and hence
$$(s_{k_{l-1}+1}s_{k_{l-1}+2}\cdots s_{k_l})\cdots (s_{k_1+1}s_{k_1+2}\cdots s_{k_2})(s_{k_0}s_{k_0+1}\cdots s_{k_1})\le w_d.$$
Since $\Sigma$ is arbitrary, it follows that $w_d=w'_d$.
\end{rema}

\subsection{An Example}
\label{example}
In this subsection, we give an example.
We consider the case for $n=5, m=12$ and $\mu=(4,3,3,2,0)$.
Then $\mu_1=(1,0,0,0,0),\mu_2=(1,1,1,0,0),\mu_3=(1,1,1,1,0),\mu_4=(1,1,1,1,0),\ld_b=(2,2,3,2,3)$ and $\ld_b^{\op}=(3,2,3,2,2)$.
Set $$\bb=\young(1133,224,345,55)\in \B_{\mu}(\ld_b).$$
Then
\begin{align*}
\FE&(\bb)=\bb_1\otimes \bb_2\otimes \bb_3\otimes \bb_4=\young(3)\otimes \young(3,4,5)\otimes \young(1,2,4,5)\otimes \young(1,2,3,5)
\in \B^{\hG^d}_{\mu_\bl},
\end{align*}
and
\begin{align*}
u_2(\bb)&=\young(3)\otimes \young(3)\otimes \young(2)\otimes \young(2)\otimes \young(3)=--++-,\\
u_2(\bb)_{\red}&=--+,\ve_2(\bb)=2, \phi_2(\bb)=1,\\
u_4(\bb)&=\young(4)\otimes \young(5)\otimes \young(4)\otimes \young(5)\otimes \young(5)=+-+--,\\
u_4(\bb)_{\red}&=-, \ve_4(\bb)=1, \phi_2(\bb)=0.
\end{align*}
So by Theorem \ref{ef}, we have
\begin{align*}
\te_2 \bb =\young(1123,224,345,55), \tf_2 \bb =\young(1133,234,345,55), \te_4 \bb =\young(1133,224,345,45), \tf_4\bb=0.
\end{align*}
In a similar way, we compute
\begin{align*}
\te_2 \bb &=\young(1123,224,345,55), \te_2\te_4 \bb =\young(1123,224,345,45), \te_1\te_2 \bb =\young(1113,224,345,55),\\
\te_3\te_2\te_4 \bb &=\young(1123,224,335,45), \te_1\te_2\te_4 \bb =\young(1113,224,345,45), \te_3\te_4\te_1\te_2 \bb =\young(1113,224,335,45).
\end{align*}
By Theorem \ref{minuscule bijection}, we want to find $\ld_\bl$ satisfying
\begin{align*}
&\hspace{0.6cm}[\ld_\bl]=(\Gamma^{G^d})^{-1}(\FE(\bb))\\
&\Leftrightarrow\ld^\flat_{\bl}=\FE(\bb)\in \B^{\hG^d}_{\mu_\bl}(\ld_b)\\
&\Leftrightarrow\up_{\ld_1}^{-1}(\ld_2-\ld_1)=\wt(\bb_1)=(0,0,1,0,0),\\
&\hspace{0.66cm}\up_{\ld_2}^{-1}(\ld_3-\ld_2)=\wt(\bb_2)=(0,0,1,1,1),\\
&\hspace{0.66cm}\up_{\ld_3}^{-1}(\ld_4-\ld_3)=\wt(\bb_3)=(1,1,0,1,1),\\
&\hspace{0.66cm}\up_{\ld_4}^{-1}(b\ld_1-\ld_4)=\wt(\bb_4)=(1,1,1,0,1).
\end{align*}
In the sequel, we check that for $\up\in \Upsilon(\bb)$, $\ld_\bl=\xi_\bl(\bb,\up)$ satisfies these equations.
Since
\begin{align*}
\bb^{\op}=\te_3\te_4\te_1\te_2\bb=\young(1113,224,335,45)\in \B_{\mu}(\ld_b^{\op}),
\end{align*}
we have
\begin{align*}
\FE(\bb^{\op})=\young(3)\otimes s_1s_2\young(3,4,5)\otimes s_3\young(1,2,4,5)\otimes s_4\young(1,2,3,5)
\in \B^{\hG^d}_{\mu_\bl},
\end{align*}
and
\begin{align*}
w_1=1, w_2=s_1s_2, w_3=s_3,w_4=s_4, w(\bb)=w_1^{-1}w_2^{-1}w_3^{-1}w_4^{-1}=s_2s_1s_3s_4.
\end{align*}
So
\begin{align*}
\Upsilon(\bb)&=\{\up\in W_0\mid \up^{-1}\tau^{12}\up=s_2s_1s_3s_4\}\\
&=\{\up\in W_0\mid (1\ 3\ 5\ 2\ 4)=(\up(1)\ \up(3)\ \up(4)\ \up(5)\ \up(2))\}\\
&=\{(1\ 3\ 5\ 4\ 2),(2\ 4\ 5), (1\ 5)(2\ 3),(1\ 2\ 5\ 3\ 4),(1\ 4\ 3)\}.
\end{align*}
Set $\up_1=(1\ 3\ 5\ 4\ 2),\up_2=(2\ 4\ 5), \up_3=(1\ 5)(2\ 3),\up_4=(1\ 2\ 5\ 3\ 4),\up_5=(1\ 4\ 3)$.
Then
\begin{align*}
\up_1^{-1}\ld_b^-&=(2,2,3,2,3),  \up_2^{-1}\ld_b^-=(2,3,2,3,2),\up_3^{-1}\ld_b^-=(3,2,2,3,2),\\
\up_4^{-1}\ld_b^-&=(2,3,3,2,2),\up_5^{-1}\ld_b^-=(3,2,2,2,3).
\end{align*}
The corresponding conjugates of $\bb$ are
\begin{align*}
\bb &=\young(1133,224,345,55), \hspace{0.2cm}\te_2\te_4 \bb =\young(1123,224,345,45),\hspace{0.3cm} \te_1\te_2\te_4 \bb =\young(1113,224,345,45),\\
\te_3\te_2\te_4 \bb &=\young(1123,224,335,45),\hspace{0.2cm} \te_1\te_2\bb =\young(1113,224,345,55),
\end{align*}
respectively.
From this, we compute
\begin{align*}
\xi(\up_1^{-1}\bb^-)&=(3,3,1,1,0), \xi(\up_2^{-1}\bb^-)=(3,2,1,0,0),\xi(\up_3^{-1}\bb^-)=(2,2,1,0,0),\\
\xi(\up_4^{-1}\bb^-)&=(3,2,1,1,0),\xi(\up_5^{-1}\bb^-)=(3,3,2,1,0),
\end{align*}
and
\begin{align*}
\up_1\xi(\up_1^{-1}\bb^-)&=(3,1,3,0,1), \up_2\xi(\up_2^{-1}\bb^-)=(3,0,1,2,0),\up_3\xi(\up_3^{-1}\bb^-)=(0,1,2,0,2),\\
\up_4\xi(\up_4^{-1}\bb^-)&=(1,3,0,1,2),\up_5\xi(\up_5^{-1}\bb^-)=(2,3,1,3,0).
\end{align*}
Note that
\begin{align*}
\up_2\xi(\up_2^{-1}\bb^-)&=\eta(\up_3\xi(\up_3^{-1}\bb^-)),\up_4\xi(\up_4^{-1}\bb^-)=\eta(\up_2\xi(\up_2^{-1}\bb^-)),\\
\up_1\xi(\up_1^{-1}\bb^-)&=\eta(\up_4\xi(\up_4^{-1}\bb^-)),
\up_5\xi(\up_5^{-1}\bb^-)=\eta(\up_1\xi(\up_1^{-1}\bb^-)).
\end{align*}
We first consider the case for $\up_3$.
Set $\xi_\bl=\xi_\bl(\bb,\up_3)$.
Then
\begin{align*}
\xi_1&=(0,1,2,0,2),\\
\xi_2&=\xi_1+\up_3\wt(\bb_1)=(0,2,2,0,2),\\
\xi_3&=\xi_2+\up_3\wt(\bb_2)=(1,3,2,1,2),\\
\xi_4&=\xi_3+\up_3 s_2s_1\wt(\bb_3)=(2,4,2,2,3).
\end{align*}
We can check that
\begin{align*}
\up_{\xi_1}=\up_3,\up_{\xi_2}=\up_3=\up_3 w_1^{-1},\up_{\xi_3}=\up_3 s_2s_1=\up_3 w_1^{-1}w_2^{-1},\up_{\xi_4}&=\up_3 s_2s_1s_3=\up_3 w_1^{-1}w_2^{-1}w_3^{-1},
\end{align*}
and
\begin{align*}
b\xi_1-\xi_4 &=\tau^{12}\xi_1+(3,3,2,2,2)-\xi_4\\
&=(0,2,0,1,2)+(3,3,2,2,2)-(2,4,2,2,3)\\
&=(1,1,0,1,1)=\up_{\xi_4}\wt (\bb_4).
\end{align*}
Thus $\xi_\bl^\flat=\FE(\bb)$.
The same holds for other $\up\in \Upsilon(\bb)$ because $\up_{\eta \ld}=\tau \up_{\ld}$.

In the above example, there exists a partial Coxeter element $w_\up$ such that $\up^{-1}\ld_b^-=w_\up\ld_b$ for any $\up\in \Upsilon(\bb)$.
In fact the same is true in general, see Lemma \ref{allowed}.
Here we illustrate this for $n=5$ and $m_0=2$.
In this case there are $8$ Coxeter elements:
\begin{align*}
s_1s_2s_3s_4=(1\ 2\ 3\ 4\ 5),\quad s_2s_3s_4s_1=(1\ 3\ 4\ 5\ 2),\\
s_3s_4s_1s_2=(1\ 2\ 4\ 5\ 3),\quad s_4s_1s_2s_3=(1\ 2\ 3\ 5\ 4),\\
s_3s_4s_2s_1=(1\ 4\ 5\ 3\ 2),\quad s_4s_2s_1s_3=(1\ 3\ 5\ 4\ 2),\\
s_4s_1s_3s_2=(1\ 2\ 5\ 4\ 3),\quad s_4s_3s_2s_1=(1\ 5\ 4\ 3\ 2).
\end{align*}
Note that $1,2$ and $4,5$ are adjacent respectively in these $n$-cycles (cf.\ Lemma \ref{Coxeter}).
On the other hand, $4,5$ in $\tau^m=(1\ 3\ 5\ 2\ 4)$ are not adjacent.
Since $(\up^{-1}\ld_b^-)(i)=\ld_b^-(\up(i))$ and $\up^{-1}\tau^m\up$ is one of the Coxeter elements listed above, we have $\up^{-1}\ld_b^-\neq (\mn+1,\mn+1,\mn,\mn,\mn), (\mn,\mn,\mn,\mn+1,\mn+1).$
For other conjugate $\ld$ of $\ld_b$, there exists a partial Coxeter element $w$ such that $\ld=w\ld_b$.
Thus our claim is verified in this case.

\section{Proof of Theorem \ref{constructionthm}}
\label{proof}
Keep the notations and assumptions above.
In \S\ref{Coxeter tau}, we collect some properties of Coxeter elements and $\tau^m$.
We need these facts to study $\up^{-1}\ld_b^-$ (and hence $\up^{-1}\bb^-$) in \S\ref{Allowed Cocharacters} because $\up$ is an element defined by measuring the difference between $w(\bb)$ and $\tau^m$.
In \S\ref{computation}, we examine the relationship between $w(\bb)$ and the computation of $\ve_i(\up^{-1}\bb^-)$ from $\ve_i(\bb)$.
In \S\ref{inequalities}, we will establish some inequalities on $\ve_i(\up^{-1}\bb^-)$ from this computation.
These inequalities are the key to the proof of $\xi_l(\bb,\up)_{\up w_1^{-1}w_2^{-1}\cdots w_{l-1}^{-1}\chi_{i,j}}\geq 0$ for all $\chi_{i,j}\in \Phi_+$, see \S\ref{upsilonsubsection} for details.
By definition, this is equivalent to $\up_{\xi_l(\bb,\up)}=\up w_1^{-1}w_2^{-1}\cdots w_{l-1}^{-1}$ (cf.\ \S\ref{Irreducible Components}).
In \S\ref{end}, we finish the proof of Theorem \ref{constructionthm} using tensor structure of crystals.

\subsection{Coxeter Elements and $\tau^m$}
\label{Coxeter tau}
Every Coxeter element in the symmetric group $W_0$ is a cycle of length $n$.
The next lemma says that the numbers $1, 2,\ldots, j$ (resp.\ $n, n-1,\ldots, n-j$) appering in the cycle corresponding to a Coxeter element are ``successive''.
\begin{lemm}
\label{Coxeter}
If $w\in W_0$ is a Coxeter element, then
for any $1\le j\le n$, there exists $1\le i\le j$ (resp.\ $n-j+1\le i\le n$) such that
$\{i, w(i),\ldots, w^{j-1}(i)\}=\{1,2,\ldots, j\}$ (resp.\ $\{i, w(i),\ldots, w^{j-1}(i)\}=\{n,n-1,\ldots, n-j+1\}$).
\end{lemm}
\begin{proof}
It suffices to prove the case for $\{1,2,\ldots, j\}$.
We argue by induction on $n$.
If $n=2$, the statement is obvious.
Suppose it is true for $n-1$.
Then any Coxeter element in $W_0$ can be written as a product of $w$ and $s_{n-1}$ such that $w$ is a Coxeter element of the symmetric group of degree $n-1$.
The case for $j=n-1, n$ is obvious.
If $1\le j<n-1$, then by the induction hypothesis, it is easy to check that the statement holds for both $ws_{n-1}$ and $s_{n-1}w$.
This completes the proof.
\end{proof}

Let $w$ be a Coxeter element in $W_0$.
Fix a reduced expression $w=s_{j_1}s_{j_2}\cdots s_{j_{n-1}}$.
Then $s_{j_h}s_{j_h-1}\le w$ (resp.\ $s_{j_h}s_{j_h+1}\le w$) if and only if $j_{h'}=j_h-1$ (resp.\ $j_h+1$) for some $h'>h$.
\begin{coro}
\label{Coxeter2}
Let $w$ be a Coxeter element in $W_0$.
Set $s_0=s_n=1$.
\begin{enumerate}[(i)]
\item If $s_{i}s_{i-1}\le w$ and $s_{i}s_{i+1}\le w$, then $$\{i,w(i),\ldots, w^i(i)\}=\{1,2,\ldots,i+1\}$$ and $w^i(i)=i+1$.

\item If $s_{i-1}s_i\le w$ and $s_{i}s_{i+1}\le w$, then $w(i)=i+1$.

\item If $s_{i}s_{i-1}\le w$ and $s_{i+1}s_i\le w$, then $w(i+1)=i$.

\item If $s_{i-1}s_i\le w$ and $s_{i+1}s_i\le w$, then $$\{i+1,w(i+1),\ldots, w^i(i+1)\}=\{1,2,\ldots,i+1\}$$ and $w^i(i+1)=i$.
\end{enumerate}
\end{coro}

\begin{proof}
If $s_{i}s_{i-1}\le w$ and $s_{i}s_{i+1}\le w$, then $w^{-1}(i)\geq i+1$.
So by Lemma \ref{Coxeter}, we have $\{i,w(i),\ldots, w^{i-1}(i)\}=\{1,2,\ldots,i\}$.
Moreover if $i<n-1$, then $w^{-1}(i)> i+1$.
Again by Lemma \ref{Coxeter}, we have $w^i(i)=i+1$.
This proves (i).
Note that $s_j$ with $j\neq i-1,i,i+1$ does not affect $i, i+1$.
The assertion of (ii) follows immediately from this.
The proof of (iii) and (iv) is similar.
\end{proof}

The following facts on $\tau^m$ are also useful.
\begin{lemm}
\label{tau}
Let $1\le r<m_0$ be the residue of $n$ modulo $m_0$.
\begin{enumerate}[(i)]
\item We have $\{\tau^{(i_1-1)m}(1),\tau^{(i_2-1)m}(1),\ldots, \tau^{(i_{m_0}-1)m}(1)\}=\{n-m_0+1,n-m_0+2,\ldots, n\}$.
\item For any $1\le k\le m_0-1$, $\tau^{(i_k-1)m}(1)-\tau^{(i_{k+1}-1)m}(1)$ is congruent to $n$ modulo $m_0$.
This is also true for $\tau^{(i_{m_0}-1)m}(1)-\tau^{(i_1-1)m}(1)$.
\item For any $1\le k\le m_0$, $i_k-i_{k-1}$ is equal to $i_1$ or $i_1-1$ according to whether $\tau^{(i_k-1)m}(1)>n-r$ or $\tau^{(i_k-1)m}(1)\le n-r$.
\end{enumerate}
\end{lemm}
\begin{proof}
By the definition of $i_k$ (cf.\ \S\ref{construction}), we have $\tau^{im}(1)=1+im_0-(k-1)n$ for $i_{k-1}\le i<i_k$.
The assertion of (ii) follows immediately from this.
Note that $\tau^{im}(1)>n-m_0$ if and only if $\tau^{(i+1)m}(1)<\tau^{im}(1)$.
This implies (i).

Fix $k$.
By (i), $\tau^{(i_k-1)m}(1)>n-m_0$ and $1\le \tau^{(i_k-1-j)m}(1)=\tau^{(i_k-1)m}(1)-jm_0\le n-m_0$ for $1\le j<i_k-i_{k-1}$.
So $i_k-i_{k-1}$ is equal to the minimal integer $i$ such that $\tau^{(i_k-1)m}(1)-im_0<0$.
Again by the definition of $i_1$, $i=i_1$ if $\tau^{(i_k-1)m}(1)=n$.
Thus $i_k-i_{k-1}=i_1$ (resp.\ $i_1-1$) if and only if $\tau^{(i_k-1)m}(1)>n-r$ (resp.\ $\tau^{(i_k-1)m}(1)\le n-r$).
\end{proof}

In below, let $X_{>a}$ denote the set $\{x\in X\mid x> a\}$ for a set $X\subset \Z$ and an integer $a$.
The following two lemmas will be used in \S\ref{Allowed Cocharacters}.
\begin{lemm}
\label{n-r}
Let $1\le r<m_0$ be the residue of $n$ modulo $m_0$.
Fix $2\le k\le m_0$ and let $z_k\in \{n-m_0+1, n-m_0+2,\ldots, n\}$.
We define $z_1,\ldots, z_{k-1}\in \{n-m_0+1, n-m_0+2,\ldots, n\}$ such that $z_1-z_2, \ldots, z_{k-1}-z_k$ are congruent to $n$ modulo $m_0$.
Then
$$|\{z_1,z_2,\ldots, z_k\}_{>n-r}|\geq |\{\tau^{(i_{m_0-k+1}-1)m}(1),\tau^{(i_{m_0-k+2}-1)m}(1),\ldots, \tau^{(i_{m_0}-1)m}(1)\}_{>n-r}|.$$
\end{lemm}
\begin{proof}
For an integer $a$, let $n-m_0+1\le [a]_{m_0}\le n$ denote its residue modulo $m_0$.
Set $Z(z_k)=\{z_1,z_2,\ldots, z_k\}$.
If $z_k<n$, then we have
$$|\{[z_1+1]_{m_0},[z_2+1]_{m_0},\ldots, [z_k+1]_{m_0}\}_{>n-r}|\geq |Z(z_k)_{>n-r}|.$$
This is obvious if $n\notin Z(z_k)$.
If $z_l=n$, then $l<k$ and $[z_{l+1}+1]_{m_0}=z_{l+1}+1=n-r+1$.
Thus the inequality holds.
Note that $\tau^{(i_{m_0}-1)m}(1)=n-m_0+1$.
So by Lemma \ref{tau} (ii), we have $Z(n-m_0+1)=\{\tau^{(i_{m_0-k+1}-1)m}(1),\tau^{(i_{m_0-k+2}-1)m}(1),\ldots, \tau^{(i_{m_0}-1)m}(1)\}.$
Combining this with the above inequality, we obtain the lemma.
\end{proof}

\begin{lemm}
\label{negative}
Let $1\le k\le m_0$ and let $i_{k-1}<j\le i_k$.
\begin{enumerate}[(i)]
\item Let $1\le z\le n$ such that $\tau^{(j-1)m}(z)\le n-m_0$.
Then
\begin{align*}
|\{z, \tau^m(z),\ldots, \tau^{(j-1)m}(z)\}_{>n-m_0}|=k&\Leftrightarrow z>\tau^{(j-1)m}(z),\\
|\{z, \tau^m(z),\ldots, \tau^{(j-1)m}(z)\}_{>n-m_0}|=k-1&\Leftrightarrow z<\tau^{(j-1)m}(z).
\end{align*}
\item Let $1\le z\le n$ such that $\tau^{(j-1)m}(z)> n-m_0$.
Then 
\begin{align*}
|\{z, \tau^m(z),\ldots, \tau^{(j-1)m}(z)\}_{>n-m_0}|=k+1&\Leftrightarrow z>\tau^{(j-1)m}(z),\\
|\{z, \tau^m(z),\ldots, \tau^{(j-1)m}(z)\}_{>n-m_0}|=k&\Leftrightarrow z<\tau^{(j-1)m}(z).
\end{align*}
\end{enumerate}
\end{lemm}
\begin{proof}
By Lemma \ref{tau} (i), $\tau^{im}(1)>n-m_0$ if and only if $i=i_k-1$ for some $1\le k\le m_0$.
So if $i_{k-1}<j< i_k$ (resp.\ $j=i_k$), we have 
$$
\text{$|\{1, \tau^m(1),\ldots, \tau^{(j-1)m}(1)\}_{>n-m_0}|=k-1$ (resp.\ $k$).}
$$
For $1\le z\le n-1$, set $Z=\{1, \tau^m(z),\ldots, \tau^{(j-1)m}(z)\}$.
For an integer $a$, let $1\le [a]_n\le n$ denote its residue modulo $n$.
If $\tau^{(j-1)m}(z)\neq n-m_0$ (resp.\ $\tau^{(j-1)m}(z)=n-m_0$), then 
$|\{[z+1]_n, [\tau^m(z)+1]_n,\ldots, [\tau^{(j-1)m}(z)+1]_n\}_{>n-m_0}|=|Z_{>n-m_0}|$ (resp.\ $|Z_{>n-m_0}|+1$).
Note that $z=[1+(z-1)]_n, \tau^m(z)=[\tau^m(1)+(z-1)]_n,\ldots, \tau^{(j-1)m}(z)=[\tau^{(j-1)m}(1)+(z-1)]_n$.
Thus, as in the proof of Lemma \ref{n-r}, we can verify the lemma by adding $1$ to $\{1, \tau^m(1),\ldots, \tau^{(j-1)m}(1)\}$ repeatedly.
\end{proof}

\subsection{Allowed Cocharacters}
\label{Allowed Cocharacters}
Let $\ld$ be a conjugate of $\ld_b$.
We say $\ld$ is {\it allowed} if there exists a partial Coxeter element $w$ such that $w$ has a reduced expression $s_{j_1}s_{j_2}\cdots s_{j_h}$ satisfying $\la \chi_{j_h,j_h+1}, \ld_b\ra=-1, \la \chi_{j_{h-1},j_{h-1}+1}, s_{j_h}\ld_b\ra=-1,\ldots, \la \chi_{j_1,j_1+1}, s_{j_2}\cdots s_{j_h}\ld_b\ra=-1$ and $\ld=w\ld_b$.
This means that $\ld$ is obtained from $\ld_b$ by multiplying each simple reflection at most once and moving $\mn+1$ from right to left.
For allowed $\ld$, such $w$ is unique, and the same holds for any reduced expression of $w$.
We call this $w$ the partial Coxeter element associated to $\ld$.
In below, let $c_i$ (resp.\ $c'_i$) denote the cardinality of the set $\{j\mid 1\le j\le i, \ld(j)=\mn+1\}$ (resp.\ $\{j\mid i\le j\le n,\ld(j)=\mn+1\}$).
\begin{lemm}
\label{allowed c}
Let $\ld$ be a conjugate of $\ld_b$.
Then $\ld$ is allowed if and only if $c_{i_k-1}\le k$ and $c'_{i_{m_0-k}+1}\le k$ for all $1\le k\le m_0$.
Fix $1\le k\le m_0$.
If $\ld$ is allowed, then for $i_{k-1}\le i< i_k$, $i\in \supp(w)$ if and only if $c_i=k$, where $w$ is the partial Coxeter element associated to $\ld$.
\end{lemm}
\begin{proof}

Note that if $\ld=\ld_b$, then $c_{i_k-1}=k-1$ and $c'_{i_{m_0-k}+1}=k-1$ for any $k$.
If $\ld=w\ld_b$ is allowed, then these integers increase at most once by $w$.
So we have  $c_{i_k-1}\le k$ and $c'_{i_{m_0-k}+1}\le k$ for all $k$.
Conversely, if $c_{i_k-1}\le k$ and $c'_{i_{m_0-k}+1}\le k$ for any $k$, then in particular, we have $c_{i_k-1}\le k$ and $c'_{i_k+1}\le m_0-k$.
The latter implies that $k\le c_{i_k}$.
So we deduce that $c_{i_k-1}=k-1$ or $k$.
If $c_{i_k-1}=k-1$, set $t_k=1$.
If $c_{i_k-1}=k$, then $\{j\mid i_{k-1}\le j<i_k, \ld(j)=\mn+1 \}$ is non-empty, and contains at most two elements. 
Let $j_k$ be the greater one among them, and set $t_k=s_{j_k}s_{j_k+1}\cdots s_{i_k-1}$.
It is easy to check that $\ld=t_{m_0}\cdots t_2t_1\ld_b$ and $\ld$ is allowed as desired.

Fix $1\le k\le m_0$ and assume that $\ld$ is allowed.
For $i_{k-1}\le i< i_k$, we have $c_{i_{k-1}}\le c_i\le c_{i_k-1}$.
By the above discussion, we have $c_i=k-1$ or $k$.
Since $c_i=k-1$ if $\ld=\ld_b$, the last assertion follows immediately from the definition of allowed cocharacters.
\end{proof}

\begin{lemm}
\label{allowed}
Let $\up\in W_0$ such that $\up^{-1}\tau^m\up$ is a Coxeter element.
Then $\up^{-1}\ld_b^-$ is allowed.
\end{lemm}

The strategy of the proof is the same as the case for $n=5$ in \S\ref{example}.
The key observations are the following:
As a $n$-cycle, the numbers in a Coxeter element are successive (Lemma \ref{Coxeter}).
On the other hand, the numbers greater than $n-m_0$ in $\tau^m$ are apart enough (Lemma \ref{tau} (iii) and Lemma \ref{n-r}).
\begin{proof}
Set $\ld=\up^{-1}\ld_b^-$, and let $c_i$ be as above.
By Lemma \ref{allowed c}, we need to show that if $\up^{-1}\tau^m\up$ is a Coxeter element, then $c_{i_k-1}\le k$ and $c'_{i_{m_0-k}+1}\le k$ for all $1\le k\le m_0$.
For this, it suffices to show that for any $k$ and $1\le z\le n$, there are at most $k$ elements greater than $n-m_0$ among $z, \tau^m(z),\ldots, \tau^{(i_k-2)m}(z)$.
Indeed, by Lemma \ref{Coxeter} and the assumption that $w\coloneqq\up^{-1}\tau^m\up$ is a Coxeter element, there exists $j$ (resp.\ $j'$) such that $\{j, w(j),\ldots, w^{i_k-2}(j)\}=\{1,2,\ldots,i_k-1\}$ (resp.\ $\{j', w(j'),\ldots, w^{i_k-2}(j')\}=\{n,n-1,\ldots,n-i_k+2\}$) and
\begin{align*}
\tau^m=\up w\up^{-1}=(\cdots \up(j)\ \up(w(j))\ \cdots\ \up(w^{i_k-2}(j))\ \cdots)
\end{align*}
Since $\ld(i)=\ld_b^-(\up(i))$ and $n-i_k+2=i_{m_0-k}+1$, both $c_{i_k-1}$ and $c'_{i_{m_0-k}+1}$ are equal to the number of integers greater than $n-m_0$ appearing in $z, \tau^m(z),\ldots, \tau^{(i_k-2)m}(z)$ for some $z$.

If $\tau^{(i_k-2)m}(z)\le n-m_0$, then we have
$$|\{\tau^{-m}(z), z,\ldots, \tau^{(i_k-3)m}(z)\}_{>n-m_0}|\geq|\{z, \tau^m(z),\ldots, \tau^{(i_k-2)m}(z)\}_{>n-m_0}|.$$
So we may replace $z$ by $\tau^{-m}z$.
Repeating this, we may assume $\tau^{(i_k-2)m}(z)>n-m_0$.
It follows from Lemma \ref{tau} (i) and (iii) that if $j$ is the minimal positive integer such that $\tau^{-jm}(z')>n-m_0$ for some $n-m_0< z'\le n$, then $j=i_1$ or $i_1-1$ according to whether $z'>n-r$ or $z'\le n-r$.
So in particular, our claim is true for $k=1$.
To show the case for $2\le k\le m_0$, we argue by contradiction.
Suppose $|\{z, \tau^m(z),\ldots, \tau^{(i_k-2)m}(z)\}_{>n-m_0}|>k$.
Set $z_k=\tau^{(i_k-2)m}(z)$ and define $z_1,\ldots, z_{k-1}$ as in Lemma \ref{n-r}.
By Lemma \ref{tau} (i), we have $z_{j-1}=\tau^{-i_1m}(z_j)$ or $\tau^{-(i_1-1)m}(z_j)$ for $2\le j\le k$.
Set $$z_0=
\begin{cases}
 \tau^{-(i_1-1)m}(z_1) & (\text{if}\ \tau^{-(i_1-1)m}(z_1)>n-m_0)\\
 \tau^{-i_1 m}(z_1)& (\text{if}\ \tau^{-(i_1-1)m}(z_1)\le n-m_0).
\end{cases}$$
Then $z_0>n-m_0$.
By $|\{z, \tau^m(z),\ldots, \tau^{(i_k-2)m}(z)\}_{>n-m_0}|>k$, we have $$\{z_0,z_1,\ldots,z_k\}\subseteq\{z, \tau^m(z),\ldots, \tau^{(i_k-2)m}(z)\}_{>n-m_0}.$$
So by Lemma \ref{tau} (iii) and Lemma \ref{n-r}, we have 
\begin{align*}
&\hspace{0.561cm}|\{z, \tau^m(z),\ldots, \tau^{(i_k-2)m}(z)\}|\\
&\geq |\{z_0,\tau^m(z_0),\tau^{2m}(z_0),\ldots,z_1, \ldots, \tau^m(z_{k-1}),\tau^{2m}(z_{k-1}),\ldots,z_k\}| \\
&\geq |\{\tau^{(i_{m_0-k}-1)m}(1),\tau^{i_{m_0-k}m}(1),\ldots, \tau^{(i_{m_0}-1)m}(1)\}|=n-i_{m_0-k}+1=i_k,
\end{align*}
which is a contradiction.
Therefore there are at most $k$ elements greater than $n-m_0$ among $z, \tau^m(z),\ldots, \tau^{(i_k-2)m}(z)$ for any $z$.
This completes the proof.
\end{proof}

Let $\bb\in \B_\mu(\ld_b)$ and $\up\in \Upsilon(\bb)$.
Then $\up^{-1}\ld_b^-$ is allowed by Lemma \ref{allowed}.
We denote by $w_\up$ the partial Coxeter element associated to $\up^{-1}\ld_b^-$.
By Lemma \ref{Weyl action}, we can compute $\up^{-1}\bb^-$ from $\bb$ by $w_\up$.
The following corollary will be used frequently in \S\ref{inequalities}.
\begin{coro}
\label{wbCoxeter}
Set $s_0=s_n=1$.
Assume that $\la \chi_{i,i+1},\up^{-1}\ld_b^-\ra=0$.
\begin{enumerate}[(i)]
\item Assume that $s_{i}s_{i-1}\le w(\bb)$ and $s_{i}s_{i+1}\le w(\bb)$. 
Then $\up\chi_{i,i+1}\in \Phi_-$ if and only if $i\in \supp(w_\up)$.

\item Assume that $s_{i-1}s_i\le w(\bb)$ and $s_{i}s_{i+1}\le w(\bb)$.
Then $\up\chi_{i,i+1}\in \Phi_-$ if and only if $\up(i),\up(i+1)>n-m_0$.

\item Assume that $s_{i}s_{i-1}\le w(\bb)$ and $s_{i+1}s_i\le w(\bb)$. 
Then $\up\chi_{i,i+1}\in \Phi_-$ if and only if $\up(i),\up(i+1)\le n-m_0$.

\item Assume that $s_{i-1}s_i\le w(\bb)$ and $s_{i+1}s_i\le w(\bb)$.
Then $\up\chi_{i,i+1}\in \Phi_-$ if and only if $i\notin \supp(w_\up)$.
\end{enumerate}
\end{coro}
\begin{proof}
By the definition of $\up$, we have $\up w(\bb)\up^{-1}=\tau^m$.
So for any $j$, there exists $1\le z\le n$ such that $\up(i)=z,\up(w(\bb)(i))=\tau^m(z),\ldots, \up(w(\bb)^{j-1}(i))=\tau^{(j-1)m}(z)$.
The assertions of (ii) and (iii) follow immediately from this and Corollary \ref{Coxeter2}.

It remains to prove (i) and (iv).
We only prove (i), and the proof of (iv) is similar.
Assume that $\up(i),\up(i+1)\le n-m_0$ (resp.\ $\up(i),\up(i+1)>n-m_0$).
Let $k$ such that $i_{k-1}\le i<i_k$.
By Corollary \ref{Coxeter2} (i) and Lemma \ref{negative}, $|\{\up(1),\ldots,\up(i+1)\}_{>n-m_0}|=k$ (resp.\ $k+1$) if and only if $\up(i)>\up(i+1)$.
By Corollary \ref{Coxeter2} (i) and Lemma \ref{allowed c}, $|\{\up(1),\ldots,\up(i+1)\}_{>n-m_0}|=k$ (resp.\ $k+1$) if and only if $i\in \supp(w_\up)$.
This proves (i).
\end{proof}

For $\supp(w_\up)$, we also have the following lemma:
\begin{lemm}
\label{suppw}
Let $\bb\in \B_\mu(\ld_b)$ and $\up\in \Upsilon(\bb)$.
Fix a reduced expression $s_{j_1}s_{j_2}\cdots s_{j_{n-1}}$ of $w(\bb)$.
For $1\le h\le n-1$, $j_h\in \supp(w_\up)$ if and only if $\up s_{j_1}s_{j_2}\cdots s_{j_{h-1}}\chi_{j_h,j_h+1}\in \Phi_-$.
\end{lemm}
\begin{proof}
Set $\ld=\up^{-1}\ld_b^-$, and let $c_i$ be as above.
Assume $i_{k-1}\le j_h<i_k$.
By Lemma \ref{allowed c}, $j_h\in \supp(w_\up)$ if and only if $c_{j_h}=k$.

Set $\underline{j}=s_{j_1}s_{j_2}\cdots s_{j_{h-1}}(j_h)(\le j_h)$ and $\overline{j}=s_{j_1}s_{j_2}\cdots s_{j_{h-1}}(j_h+1)(\geq j_h+1)$.
Since $w(\bb)^{-1}(\underline{j})\geq j_h+1$ and $w(\bb)^{-1}(\overline{j})\le j_h$, we have
$$\{\underline{j}, w(\bb)(\underline{j}), \ldots, w(\bb)^{j_h-1}(\underline{j})\}=\{1,2,\ldots, j_h\}$$
and $w(\bb)^{j_h}(\underline{j})=\overline{j}$ by Lemma \ref{Coxeter}.
If $\ld(\overline{j})=\ld_b^-(\up(\overline{j}))=\mn$ (resp.\ $\mn+1$), then $|\{\up(1),\up(2),\ldots,\up(j_h)\}_{>n-m_0}|=k$ is equivalent to $|\{\up(1),\up(2),\ldots,\up(j_h),\up(\overline{j})\}_{>n-m_0}|=k$ (resp.\ $k+1$).
Thus by $\up w(\bb)\up^{-1}=\tau^m$ and Lemma \ref{negative} (i) (resp.\ (ii)) for $j=j_h+1$, 
$c_{j_h}=k$ if and only if $\up(\underline{j})>\up(\overline{j})$, i.e., $\up s_{j_1}s_{j_2}\cdots s_{j_{h-1}}\chi_{j_h,j_h+1}\in \Phi_-$.
\end{proof}

\begin{coro}
\label{suppwcoro}
Keep the notation in Lemma \ref{suppw}.
Set $s_0=s_n=1$.
\begin{enumerate}[(i)]
\item Assume that $s_{j_h}s_{j_h-1}\le w(\bb)$ and $s_{j_h}s_{j_h+1}\le w(\bb)$ for fixed $h$.
Assume further that there exists $h<h'$ such that $j_{h'}=j_h-1$ (resp.\ $j_h+1$) and $j_h+1\notin \{j_1,j_2,\ldots, j_{h'-1}\}$ (resp.\ $j_h-1\notin \{j_1,j_2,\ldots, j_{h'-1}\}$).
If $\ld_b^-(\up(j_h))=\mn$ (resp.\ $\ld_b^-(\up(j_h+1))=\mn+1$), then $\up s_{j_1}\cdots s_{j_{h'}}\chi_{j_h,j_h+1}\in \Phi_-$.

\item Assume that $s_{j_h-1}s_{j_h}\le w(\bb)$ and $s_{j_h+1}s_{j_h}\le w(\bb)$ for fixed $h$.
Assume further that there exists $h'<h$ such that $j_{h'}=j_h-1$ (resp.\ $j_h+1$) and $j_h+1\notin \{j_1,j_2,\ldots, j_{h'-1}\}$ (resp.\ $j_h-1\notin \{j_1,j_2,\ldots, j_{h'-1}\}$).
If $\ld_b^-(\up(j_h+1))=\mn$ (resp.\ $\ld_b^-(\up(j_h))=\mn+1$), then $\up s_{j_1}\cdots s_{j_{h'}}\chi_{j_h,j_h+1}\in \Phi_-$.

\item Assume that $s_{j_h}s_{j_h-1}\le w(\bb)$ and $s_{j_h}s_{j_h+1}\le w(\bb)$ for fixed $h$.
If $\la \chi_{j_h,j_h+1}, \up^{-1}\ld_b^-\ra=-1$, then $j_h\notin \supp(w_\up)$.

\item Assume that $s_{j_h-1}s_{j_h}\le w(\bb)$ and $s_{j_h+1}s_{j_h}\le w(\bb)$ for fixed $h$.
If $\la \chi_{j_h,j_h+1}, \up^{-1}\ld_b^-\ra=1$, then $j_h\in \supp(w_\up)$.
\end{enumerate}
\end{coro}
\begin{proof}
Keep the notation in the proof of Lemma \ref{suppw}.

Set $j=s_{j_{n-1}}\cdots s_{j_{h'+1}}(j_{h}+1)$ (resp.\ $s_{j_{n-1}}\cdots s_{j_{h'+1}}(j_{h})$).
Then, by $\up w(\bb)=\tau^m \up$, $\up s_{j_1}\cdots s_{j_{h'}}\chi_{j_h,j_h+1}=\up w(\bb)\chi_{j_h, j}=\tau^m \up\chi_{j_h, j}$ (resp.\ $\tau^m \up\chi_{j, j_h+1}$) and $w(\bb)(j)=j_h$ (resp.\ $w(\bb)(j)=j_h+1$).
Moreover, if $\ld_b^-(\up(j_h))=\mn$ (resp.\ $\ld_b^-(\up(j_h+1))=\mn+1$), i.e., $\up(j_h)\le n-m_0$ (resp.\  $\up(j_h+1)>n-m_0$), then again by $\up w(\bb)\up^{-1}=\tau^m$, $\up(j)\le n-m_0$ (resp.\ $\up(j)>n-m_0$) implies $\up(j)<\up(j_h)$ (resp.\ $\up(j)>\up(j_{h}+1)$).
Combining these facts, we deduce $\up s_{j_1}\cdots s_{j_{h'}}\chi_{j_h,j_h+1}\in \Phi_-$.
The proof of (i) is finished.

Recall that we have $\{\underline{j}, w(\bb)(\underline{j}), \ldots, w(\bb)^{j_h-1}(\underline{j})\}=\{1,2,\ldots, j_h\}$
and $w(\bb)^{j_h}(\underline{j})=\overline{j}$.
Combining this with Corollary \ref{Coxeter2} (iv), we can easily check that $w(\bb)(j_h+1)=\underline{j}$ and $w(\bb)(j_h)=\overline{j}$.
Thus, by $\up w(\bb)\up^{-1}=\tau^m$, if $\ld_b^-(\up(j_h+1))=\mn$ (resp.\ $\ld_b^-(\up(j_h))=\mn+1$), i.e., $\up(j_h+1)\le n-m_0$ (resp.\  $\up(j_h)>n-m_0$), we have $\up(\underline{j})>\up(j_h+1)$ (resp.\ $\up(j_h)>\up(\overline{j})$).
By our assumption on $h'$, this is equivalent to $\up s_{j_1}\cdots s_{j_{h'}}\chi_{j_h,j_h+1}\in \Phi_-$.
The proof of (ii) is finished.

For (iii), by $s_{j_h}s_{j_h-1}\le w(\bb), s_{j_h}s_{j_h+1}\le w(\bb)$ and Lemma \ref{suppw}, $j_h\in \supp(w_\up)$ if and only if $\up(j_h)>\up(j_h+1)$.
Further, $\la \chi_{j_h,j_h+1}, \up^{-1}\ld_b^-\ra=-1$ implies $\up(j_h)\le n-m_0<\up(j_h+1)$.
Thus $j_h\notin \supp(w_\up)$.
The proof of (iii) is finished.

For (iv), by $s_{j_h-1}s_{j_h}\le w(\bb)$ and $s_{j_h+1}s_{j_h}\le w(\bb)$, we have $\underline{j}=w(\bb)(j_h+1)$ and $\overline{j}=w(\bb)(j_h)$.
So by Lemma \ref{suppw}, $j_h\in \supp(w_\up)$ if and only if $\up w(\bb)\chi_{j_h+1,j_h}\in \Phi_-$.
By $\up w(\bb)=\tau^m \up$, this is equivalent to saying  $\tau^m\up\chi_{j_h+1,j_h}\in \Phi_-$.
This holds if $\la \chi_{j_h,j_h+1}, \up^{-1}\ld_b^-\ra=1$, i.e., $\up(j_h+1)\le n-m_0<\up(j_h)$.
Thus $j_h\in \supp(w_\up)$.
The proof of (iv) is finished.
\end{proof}

\subsection{Computation of Kashiwara Operators}
\label{computation}
As explained in \S \ref{construction}, we can compute $\bb^{\op}$ from $\bb$ using each simple reflection exactly once.
Consider $u_i(\bb)$ defined in Theorem \ref{ef}.
In this computation, the action of $s_i$ changes some $-$ to $+$,
and the action of $s_{i-1}$ (resp.\ $s_{i+1}$) deletes $+$ (resp.\ adds $-$).
Other simple reflections do not affect $u_i(\bb)$ (and hence $\ve_i(\bb)$).
Let $\bb\in \B_\mu(\ld_b)$ and $\up\in \Upsilon(\bb)$.
Since $\up^{-1}\ld_b^-$ is allowed, we can use a part of this computation to obtain $\up^{-1}\bb^-$ from $\bb$ by $w_\up$.
Let $\ld$ be an allowed conjugate of $\ld_b$, and let $\bb'$ be the conjugate of $\bb$ with weight $\ld$.
Let $w$ be the partial Coxeter element associated to $\ld$.
Assume $\supp(w)\subseteq\supp(w_\up)$.
Then $\ld$ is a weight appearing in the computation of $\up^{-1}\bb^-$ from $\bb$.
If $i\in\supp(w_\up)\setminus \supp(w)$ and $\la \chi_{i,i+1},\ld\ra=-1$, then $\ve_i(s_i\bb')=\ve_i(\te_i\bb')=\ve_i(\bb')-1$ (and hence $\phi_i(s_i\bb')=\phi_i(\bb')+1$ by Definition \ref{crystaldefi} (iii)).
For the action of $s_{i-1}$ or $s_{i+1}$, we have the following lemma.
\begin{lemm}
\label{wbve}
Let $\ld,\bb'$ and $w$ be as above.
Assume that $\supp(w_\up)\setminus \supp(w)$ contains $i-1$ (resp.\ $i+1$) and $\la \chi_{i-1,i},\ld\ra=-1$ (resp.\ $\la \chi_{i+1, i+2}, \ld\ra=-1$).
If $s_is_{i-1}\le w(\bb)$ (resp.\ $s_is_{i+1}\le w(\bb)$), then $\ve_i(s_{i-1}\bb')=\ve_i(\bb')$ (resp.\ $\ve_i(s_{i+1}\bb')=\ve_i(\bb')$).
Moreover, the converse holds if $(\ld_b(i),\ld_b(i+1))\neq(\mn+1,\mn)$.
\end{lemm}

Note that $\ve_i(s_{i-1}\bb'),\ve_i(s_{i+1}\bb')\in\{\ve_i(\bb'),\ve_i(\bb')+1\}$ (and $\phi_i(s_{i-1}\bb'),\phi_i(s_{i+1}\bb')\in\{\phi_i(\bb')-1,\phi_i(\bb')\}$) in any case.
Roughly speaking, this says that $w(\bb)$ determines $\ve_i(s_{i-1}\bb')$ or $\ve_i(s_{i+1}\bb')$.
Before beginning the proof, let us illustrate why this lemma holds by an example for the case $\ld=\ld_b$.

\begin{exam}
\label{wbve example}
Assume that $\mn=7$ and $(\ld_b(i-1),\ld_b(i),\ld_b(i+1))=(7,8,8)$.
We can easily find $\mu$ and $\bb\in \B_\mu(\ld_b)$ such that $$u_i(\bb)=-++--+--/++-+++--,\quad u_i(\bb)_{\red}=--++.$$
Here $(-++--+--)_{\red}=--$ and $(++-+++--)_{\red}=++$.
If the action of $s_{i-1}$ deletes $+$ on the left (resp.\ right) of /, then $\ve_i(s_{i-1}\bb)=\ve_i(\bb)+1=3$ (resp.\ $\ve_i(s_{i-1}\bb)=\ve_i(\bb)=2$).
Let $u=-$ be the rightmost $-$ to $+$ in $u_i(\bb)_{\red}$, or equivalently, the unique $-$ in $u_i(\bb)$ adjacent to $/$.
Note that if we apply $s_i$ on $s_{i-1}\bb$, then $u$ changes to $+$.
So the action of $s_{i-1}$ deletes $+$ on the left (resp.\ right) of / if and only if $s_{i-1}s_i\le w(\bb)$ (resp.\ $s_is_{i-1}\le w(\bb)$).

We next consider $\bb\in \B_\mu(\ld_b)$ such that $$u_i(\bb)=-++--+--/+-++++--,\quad u_i(\bb)_{\red}=--++.$$
In this case, $\ve_i(s_{i-1}\bb)=\ve_i(\bb)+1$ if and only if $s_{i-1}$ deletes $+$ on the left of $-$, or the unique $+$ adjacent to /.
Nevertheless, the equivalence $\ve_i(s_{i-1}\bb)=\ve_i(\bb)+1\Leftrightarrow s_{i-1}s_i\le w(\bb)$ (and hence $\ve_i(s_{i-1}\bb)=\ve_i(\bb)\Leftrightarrow s_is_{i-1}\le w(\bb)$) still holds.
Indeed, if $s_{i-1}$ deletes the unique $+$ adjacent to /, then the action of $s_i$ on $s_{i-1}\bb$ changes $-$ next to this $+$.
So we still have $s_{i-1}s_i\le w(\bb)$, which implies the equivalence.

Assume that $\mn=7$ and $(\ld_b(i-1),\ld_b(i),\ld_b(i+1),\ld_b(i+2))=(7,8,7,8)$.
Consider $\mu$ and $\bb\in \B_\mu(\ld_b)$ such that 
$$u_i(\bb)=-+\hat{+}-+-/+++--++--.$$
We also assume that the action of $s_{i-1}$ deletes $\hat{+}$ and the action of $s_{i+1}$ adds $-$ to the place where $/$ exists.
Then $\ve_i(s_{i-1}\bb)=\ve_i(s_{i+1}\bb)=\ve_i(\bb)=1$.
On the other hand, $\ve_i(s_{i-1}s_{i+1}\bb)=\ve_i(s_{i+1}s_{i-1}\bb)=\ve_i(\bb)+1=2$, $s_{i-1}s_i\le w(\bb)$ and $s_{i+1}s_i\le w(\bb)$.
The difference from the above example is that we apply both $s_{i-1}$ and $s_{i+1}$ on $\bb$ before applying $s_i$ to compute $\bb^{\op}$.
In other words, $s_is_{i-1}\le w_{\max}'$ and $s_is_{i+1}\le w_{\max}'$.
\end{exam}

We generalize the observation in Example \ref{wbve example} as follows:

\begin{lemm}
\label{right}
Let $\ld$ be a conjugate of $\ld_b$.
Let $\bb'$ be the conjugate of $\bb$ with weight $\ld$.
Assume that $\la \chi_{i-1,i},\ld\ra=-1$ (resp.\ $\la \chi_{i+1, i+2}, \ld\ra=-1$).
\begin{enumerate}[(i)]
\item We write $u_i(\bb')=u^1\cdots u^{\ld(i)+\ld(i+1)}$ (resp.\ $u_i(s_{i+1}\bb')=u^1\cdots u^{\ld(i)+\ld(i+1)+1}$).
Let $u^{\ell_{i-1}}=+$ (resp.\ $u^{\ell_{i+1}}=-$) be the box which vanishes in $u_i(s_{i-1}\bb')$ (resp.\ $u_i(\bb')$).
If $\ve_i(s_{i-1}\bb')=\ve_i(\bb')+1$ (resp.\ $\ve_i(s_{i+1}\bb')=\ve_i(\bb')+1$), then there exists $u^\ell=-$ with $\ell_{i-1}<\ell$ (resp.\ $\ell_{i+1}\le\ell$) which remains in $u_i(s_{i-1}\bb')_{\red}$ (resp.\ $u_i(s_{i+1}\bb')_{\red}$).

\item Assume that $\ve_i(\bb')>0$.
Let $u$ be the rightmost $-$ to $+$ in $u_i(\bb')_{\red}$,
and let $u'$ be the rightmost $-$ to $+$ in $u_i(s_{i-1}\bb')_{\red}$ (resp.\ $u_i(s_{i+1}\bb')_{\red}$).
If $\ve_i(s_{i-1}\bb')=\ve_i(\bb')+1$ (resp.\ $\ve_i(s_{i+1}\bb')=\ve_i(\bb')+1$), then $u=u'$ or $u'$ is on the right side of $u$.
If $\ve_i(s_{i-1}\bb')=\ve_i(\bb')$ (resp.\ $\ve_i(s_{i+1}\bb')=\ve_i(\bb')$), then $u=u'$.
\end{enumerate}
\end{lemm}
\begin{proof}
We only prove the case for $i-1$.
The case for $i+1$ follows in a similar way.

If $\ve_i(s_{i-1}\bb')=\ve_i(\bb')+1$, then there exists $u^\ell=-$ which remains in $u_i(s_{i-1}\bb')_{\red}$ but does not remain in $u_i(\bb')_{\red}$.
Note that $u^\ell=-$ with $\ell<\ell_{i-1}$ remains in $u_i(s_{i-1}\bb')_{\red}$ if and only if it remains in $u_i(\bb')_{\red}$.
So we must have $\ell_{i-1}<\ell$.
This proves (i).
Note that if $\ve_i(s_{i-1}\bb')=\ve_i(\bb')+1$ (resp.\ $\ve_i(s_{i-1}\bb')=\ve_i(\bb')$), $u(s_{i-1}\bb')_{\red}$ is obtained from $u_i(\bb')_{\red}$ by adding one $-$ (resp.\ deleting one $+$).
The statement of (ii) follows immediately from this.
\end{proof}

For $1\le i\le n-1$, let $l_i$ be a positive integer such that $i\in \supp(w_{l_i})$.
\begin{proof}[Proof of Lemma \ref{wbve}]
We only prove the case for $i-1$.
The case for $i+1$ follows in a similar way.

First assume that $\ld_b(i)=\mn$.
Then $s_is_{i-1}\le w(\bb)$ if and only if $l_i\le l_{i-1}$.
Since $\la \chi_{i-1,i}, \ld\ra=-1$ and $\ld_b(i)=\mn$, we have $i\in \supp(w)$.
We write $u_i(\bb')=u^1\cdots u^{\ld(i)+\ld(i+1)}$.
Let $u^{\ell_i}=+$ be the unique $+$ which does not exist in $u_i(\bb)$ (i.e., the box added by the action of $s_i$),
and let $u^{\ell_{i-1}}=+$ be the box which vanishes in $u_i(s_{i-1}\bb)$.
Then $l_i\le l_{i-1}$ if and only if $\ell_i\le \ell_{i-1}$.
Note that $(u^1\cdots u^{\ell_i-1})_{\red}=-\cdots -$ (resp.\ $(u^{\ell_i}\cdots u^{\ld(i)+\ld(i+1)})_{\red}=+\cdots +$) unless no $-$ (resp.\ $+$) remains.
So if $\ell_{i-1}<\ell_i$, then $\ve_i(s_{i-1}\bb')=\ve_i(\bb')+1$.
This proves the second statement.

For the first statement, we need to show that if $\ell_i\le \ell_{i-1}$, then $\ve_i(s_{i-1}\bb')=\ve_i(\bb')$.
To show this, we first check that $(u^{\ell_i+1}\cdots u^{\ld(i)+\ld(i+1)})_{\red}=+\cdots +$ unless no $+$ remains.
This claim is obviously true when $i+1\notin \supp(w)$ or $s_is_{i+1}\le w$.
If $\ld(i+1)=\mn$ and $\ld_b(i+1)=\mn$ (resp.\ $\mn+1$), then $s_is_{i+1}\le w$ (resp.\ $i+1\notin \supp(w)$), and hence the claim holds.
If $\ld(i+1)=\mn+1$, then by $i\in \supp(w)$, we must have $\ld_b(i+1)=\mn+1$ and $i+1\in \supp(w)$.
By our assumption $\supp(w)\subset \supp(w_\up)$, we also have $\ld_b^-(\up(i))=\mn,\ld_b^-(\up(i+1))=\ld(i+1)=\mn+1$.
If moreover, $-$ remains in $(u^{\ell_i+1}\cdots u^{\ld(i)+\ld(i+1)})_{\red}$, then we must have $\ell_i<\ell_{i+1}$, where $u^{\ell_{i+1}}=-$ be the box added by the action of $s_{i+1}$.
Clearly, $\ell_i<\ell_{i+1}$ implies $s_is_{i+1}\le w(\bb)$.
By Corollary \ref{suppwcoro} (iii), this and $\ell_i\le\ell_{i-1}(\Leftrightarrow s_is_{i-1}\le w(\bb))$ imply $i\notin \supp(w_\up)$.
This contradicts to $i\in \supp(w)\subset \supp(w_\up)$, which shows our claim.
If $\ell_i\le \ell_{i-1}$, then by our claim, at most one $-$ remains after we delete $u^{\ell_{i-1}}$ and then $``+-"$ in $u^{\ell_i+1}\cdots u^{\ld(i)+\ld(i+1)}$ as far as we can.
This $-$ does not contribute to $\ve_i(s_{i-1}\bb')$ because $u^{\ell_i}=+$.
Thus we have $\ve_i(s_{i-1}\bb')=\ve_i(\bb')$.

Next Assume that $\ld_b(i)=\mn+1$.
Then $s_is_{i-1}\le w(\bb)$ if and only if $l_i<l_{i-1}$.
Let $j=\min\{j'\mid i+1\le j'\le n, \ld(j')=\mn+1\}$.
Set $\bb'_0=s_{i+1}s_{i+2}\cdots s_{j-1}\bb'$.
We write $u_i(\bb'_0)=u^1\cdots u^{2\mn+2}$.
Then this is obtained from $u_i(\bb')$ by adding one $-$.
Let $u^{\ell_{i-1}}=+$ be the box which vanishes in $u_i(s_{i-1}\bb'_0)$, and let $u^{\ell_i}=-$ be the box which vanishes in $u_i(s_is_{i-1}\bb'_0)$.
Then $l_i<l_{i-1}$ if and only if $\ell_i<\ell_{i-1}$.
Note that $s_{i-1}\bb'_0=s_{i+1}s_{i+2}\cdots s_{j-1}(s_{i-1}\bb')$, and $u^{\ell_{i-1}}$ (regarded as in $u_i(\bb')$) also vanishes in $u_i(s_{i-1}\bb')$. 
So, by Lemma \ref{right} (i), if $\ve_i(s_{i-1}\bb')=\ve_i(\bb')+1$, then there exists $u^\ell=-$ with $\ell_{i-1}<\ell$ which remains in $u_i(s_{i-1}\bb'_0)_{\red}$.
Then $\ell_{i-1}<\ell\le \ell_i$, i.e., $s_{i-1}s_i\le w(\bb)$.
Thus if $s_is_{i-1}\le w(\bb)$, then $\ve_i(s_{i-1}\bb')=\ve_i(\bb')$.
The first statement is verfied.

We further assume that $(\ld_b(i),\ld_b(i+1))=(\mn+1, \mn+1)$.
Then $\ld=\ld_b$ and $\bb'_0=\bb'$.
To prove the converse, we argue by contradiction.
If $\ve_i(s_{i-1}\bb')=\ve_i(\bb')$, then by Definition \ref{crystaldefi} (iii), we have $\ve_i(\bb')=\ve_i(s_{i-1}\bb')>0$.
By Lemma \ref{right} (ii), $u^{\ell_i}$ is also the rightmost $-$ to $+$ in $u(\bb')_{\red}$.
So if moreover $\ell_{i-1}<\ell_i$, then $\ve_i(s_{i-1}\bb')=\ve_i(\bb')+1$, which is a contradiction.
This proves the second statement.
\end{proof}

\begin{rema}
\label{remawbve}
In the proof of Lemma \ref{wbve}, the assumption $\supp(w)\subset\supp(w_\up)$ is used only in the third paragraph to treat the case where $\la\chi_{i,i+1},\ld_b\ra=-1$.
So if $\la\chi_{i,i+1},\ld_b\ra\neq -1$, the lemma is true for any allowed conjugate $\ld$ with $\la \chi_{i-1,i},\ld\ra=-1$ (resp.\ $\la \chi_{i+1, i+2}, \ld\ra=-1$) such that $s_{i-1}\ld$ (resp.\ $s_{i+1}\ld$) is allowed.
\end{rema}

We need the following corollary to treat the case $(\ld_b(i),\ld_b(i+1))=(\mn+1,\mn)$.
This corollary tells us that the converse of Lemma \ref{wbve} does not hold only if $s_{i-1}s_i\le w(\bb)$ and $s_{i+1}s_i\le w(\bb)$.
\begin{coro}
\label{wbve2}
Assume that $(\ld_b(i),\ld_b(i+1))=(\mn+1,\mn)$.
Let $j_1=\min\{j'\mid i+1< j'\le n,\ld(j')=\mn+1\}$, and let $j_2=\max\{j'\mid 1\le j'< i,\ld(j')=\mn\}$.
\begin{enumerate}[(i)]
\item Assume that $s_{i+1}s_is_{i-1}\le w(\bb)$ (resp.\ $s_{i-1}s_is_{i+1}\le w(\bb)$).
Then $$\ve_i(s_{i+1}\cdots s_{j_1-1}\bb)=\ve_i(\bb)+1\quad \text{(resp.\ $\ve_i(s_{i-1}\cdots s_{j_2}\bb)=\ve_i(\bb)+1)$}.$$
\item Assume that $s_{i-1}s_i\le w(\bb)$ and $s_{i+1}s_i\le w(\bb)$.
Then $$\ve_i(s_{i-1}\cdots s_{j_2}s_{i+1}\cdots s_{j_1-1}\bb)\geq\ve_i(\bb)+1.$$
Moreover, if $\ve_i(s_{i-1}\cdots s_{j_2}s_{i+1}\cdots s_{j_1-1}\bb)=\ve_i(\bb)+1$ (resp.\ $\ve_i(\bb)+2$), then $\ve_i(s_{i+1}\cdots s_{j_1-1}\bb)=\ve_i(s_{i-1}\cdots s_{j_2}\bb)=\ve_i(\bb)$ (resp.\ $\ve_i(\bb)+1$).
\end{enumerate}
In particular, if $s_{i-1}s_i\le w(\bb)$ or $s_{i+1}s_i\le w(\bb)$, then $\ve_i(s_{i-1}\cdots s_{j_2}s_{i+1}\cdots s_{j_1-1}\bb)\geq\ve_i(\bb)+1$.
\end{coro}
\begin{proof}
If $\ve_i(\bb)=0$, then (i) follows from Definition \ref{crystaldefi} (iii) and Lemma \ref{wbve}.
If $\ve_i(\bb)>0$, let $l_0$ be the minimal integer such that $\la \chi_{i+1,i}, \wt(\bb_1)+\wt(\bb_2)+\cdots +\wt(\bb_{l_0})\ra=\ve_i(\bb)$.
If $s_is_{i-1}\le w(\bb)$ (resp.\ $s_is_{i+1}\le w(\bb)$) and $\ve_i(s_{i+1}\cdots s_{j_1-1}\bb)=\ve_i(\bb)$ (resp.\ $\ve_i(s_{i-1}\cdots s_{j_2}\bb)=\ve_i(\bb)$), then by Lemma \ref{wbve} and Lemma \ref{right} (ii), $l_i=l_0$.
However, this and $s_{i+1}s_i\le w(\bb)$ (resp.\ $s_{i-1}s_i\le w(\bb)$) imply $\ve_i(s_{i+1}\cdots s_{j_1-1}\bb)=\ve_i(\bb)+1$ (resp.\ $\ve_i(s_{i-1}\cdots s_{j_2}\bb)=\ve_i(\bb)+1$), which is a contradiction.
Thus (i) follows.

Set $\bb'=s_{i-1}\cdots s_{j_2}s_{i+1}\cdots s_{j_1-1}\bb$.
We show that if $l_{i+1}<l_i$ (resp.\ $l_{i-1}<l_i$), then $\ve_i(\bb')\geq \ve_i(\bb)+1$.
This follows from Definition \ref{crystaldefi} (iii) if $\ve_i(\bb)=0$.
If $\ve_i(\bb)>0$ and $\ve_i(\bb')=\ve_i(\bb)$, then $l_0=l_i$ by Lemma \ref{right} (ii).
However, $l_{i+1}<l_i=l_0$ (resp.\ $l_{i+1}<l_i=l_0$) implies $\ve_i(s_{i+1}\cdots s_{j_1-1}\bb)=\ve_i(\bb)+1$ (resp.\ $\ve_i(s_{i-1}\cdots s_{j_2}\bb)=\ve_i(\bb)+1$), which is a contradiction. 
This proves the claim.
Again by Lemma \ref{right} (ii), $l_i=l_{i-1}=l_{i+1}$ implies $\ve_i(\bb')=\ve_i(\bb)+1$.
Putting things together, we have proved the inequality in (ii).

Finally, we prove the ``moreover'' part in (ii).
Assume that $s_{i-1}s_i\le w(\bb)$ and $s_{i+1}s_i\le w(\bb)$.
If $\ve_i(\bb')=\ve_i(\bb)+2$, then the statement is obvious.
If $\ve_i(s_{i+1}\cdots s_{j_1-1}\bb)=\ve_i(\bb)+1$ (resp.\ $\ve_i(s_{i+1}\cdots s_{j_1-1}\bb)=\ve_i(\bb)$) and $\ve_i(s_{i-1}\cdots s_{j_2}\bb)=\ve_i(\bb)$ (resp.\ $\ve_i(s_{i-1}\cdots s_{j_2}\bb)=\ve_i(\bb)+1$), then similarly as above, we deduce $\ve_i(\bb')=\ve_i(\bb)+2$, which is a contradiction.
So the statement for the case $\ve_i(\bb')=\ve_i(\bb)+1$ follows.
This finishes the proof.
\end{proof}

\subsection{Some Inequalities on $\ve_i(\up^{-1}\bb^-)$}
\label{inequalities}
Keep the notation in \S \ref{computation}.
Let $\bb\in \B_\mu(\ld_b)$ and $\up\in \Upsilon(\bb)$.
In this subsection, we will establish some inequalities on $\ve_i(\up^{-1}\bb^-)$ using the results in \S\ref{computation}.
These inequalities are the keys to the proof of $\up_{\xi_j(\bb,\up)}=\up w_1^{-1}\cdots w_{j-1}^{-1}$.

Set $S_{l}=\la \chi_{i+1,i}, \wt(\bb_1)+\wt(\bb_2)+\cdots +\wt(\bb_{l})\ra$ for fixed $1\le i\le n-1$.
We also set $S_0=0$.
Then $S_{l}$ is the difference of the number of $-$ and $+$ in $u_i(\bb)$ which are contained in $\bb_1,\ldots,\bb_{l}$.
Thus $S_{l}\le\ve_i(\bb)$.

\begin{lemm}
\label{inequality}
We have
$S_{l} \le\ve_i(\up^{-1}\bb^-)$ for $0\le l< l_i$.
If the equality holds for some $0\le l<l_i$, then $\up\chi_{i,i+1}\in \Phi_-$.
\end{lemm}

\begin{proof}
If $\ve_i(\up^{-1}\bb^-)\geq \ve_i(\bb)$, then the inequality is obvious.
In particular, the inequality holds if $\ve_i(\bb)=0$.
If $\ve_i(\bb)> 0$, let $l_0$ be the minimal integer such that $S_{l_0}=\ve_i(\bb)$.
It follows from Lemma \ref{right} (ii) that if $\ve_i(\bb)> 0$ and $\ve_i(\up^{-1}\bb^-)=\ve_i(\bb)-1$, then $i\in \supp(w_\up)$ and $l_0=l_i$.
This implies the inequality.
Note that if the equality holds for some $0\le l<l_i$, then $\ve_i(\up^{-1}\bb^-)=\ve_i(\bb)-1$ or $\ve_i(\bb)$.
Set $\ld=\up^{-1}\ld_b^-$.

If $i\in \supp(w_\up)$ and $(\ld(i),\ld(i+1))=(\mn,\mn+1)$, then $(\ld_b(i),\ld_b(i+1))=(\mn, \mn+1)$ and $i-1, i+1\in \supp(w_\up)$.
If moreover, the equality holds for some $0\le l<l_i$, then we must have $\ve_i(\up^{-1}\bb^-)=\ve_i(\bb)-1$ because $(\ld_b(i),\ld_b(i+1))=(\mn, \mn+1)$ and hence $l_0=l_i$.
By Lemma \ref{wbve}, we have $s_is_{i-1}\le w(\bb)$ and $s_is_{i+1}\le w(\bb)$.
This contradicts to Corollary \ref{suppwcoro} (iii).
If $i\notin \supp(w_\up),(\ld(i),\ld(i+1))=(\mn,\mn+1)$ and the equality holds for some $0\le l<l_i$, then we must have $\ve_i(\up^{-1}\bb^-)=\ve_i(\bb)$.
By Definition \ref{crystaldefi} (iii), $\ve_i(\bb)=\ve_i(\up^{-1}\bb^-)>0$.
So by Lemma \ref{right} (ii), we have $l_0=l_i$ and hence $S_{l}\le \ve_i(\bb)-1$.
This is a contradiction.
Thus the equality implies $\la\chi_{i,i+1},\ld\ra=0$ or $1$.
If $\la\chi_{i,i+1},\ld\ra=1$, then $\up(i)>n-m_0\geq \up(i+1)$ and hence $\up\chi_{i,i+1}\in\Phi_-$.
It remains to treat the case where $\ld(i)=\ld(i+1)$.

If $i\in \supp(w_\up)$ and $\ld(i)=\ld(i+1)=\mn$, then $(\ld_b(i),\ld_b(i+1))=(\mn, \mn+1)$ or $(\mn,\mn)$.
In the former case, we have $i-1\in \supp(w_\up)$ and $i+1\notin \supp(w_\up)$.
If the equality holds for some $0\le l<l_i$, then $\ve_i(\up^{-1}\bb^-)=\ve_i(\bb)-1$.
It follows from Lemma \ref{wbve} that $s_is_{i-1}\le w(\bb)$.
In the latter case, we have $i-1, i+1\in \supp(w_\up)$.
It follows from Lemma \ref{wbve} that if the equality holds for some $0\le l<l_i$ and $\ve_i(\up^{-1}\bb^-)=\ve_i(\bb)-1$, then $s_is_{i-1}\le w(\bb)$ and $s_is_{i+1}\le w(\bb)$.
If the equality holds for some $0\le l<l_i$ and $\ve_i(\up^{-1}\bb^-)=\ve_i(\bb)$, then by Lemma \ref{wbve} and Lemma \ref{right} (ii) (or Definition \ref{crystaldefi} (iii) if $\ve_i(\bb)=0$), we must have $s_{i+1}s_i\le w(\bb)$ and hence $s_is_{i-1}\le w(\bb)$.
Thus the equality for some $0\le l<l_i$ implies $s_is_{i-1}\le w(\bb)$.
Then $\up\chi_{i,i+1}\in \Phi_-$ follows from Corollary \ref{wbCoxeter} (i) and (iii).

If $i\in \supp(w_\up)$ and $\ld(i)=\ld(i+1)=\mn+1$, then $(\ld_b(i),\ld_b(i+1))=(\mn, \mn+1)$ or $(\mn+1,\mn+1)$.
In the former case, we have $i+1\in \supp(w_\up)$ and $i-1\notin \supp(w_\up)$.
If the equality holds for some $0\le l<l_i$, then $\ve_i(\up^{-1}\bb^-)=\ve_i(\bb)-1$.
It follows from Lemma \ref{wbve} that $s_is_{i+1}\le w(\bb)$.
In the latter case, we have $i-1, i+1\in \supp(w_\up)$.
It follows from Lemma \ref{wbve} and Lemma \ref{right} (ii) (or Definition \ref{crystaldefi} (iii) if $\ve_i(\bb)=0$) that if the equality holds for some $0\le l<l_i$ and $\ve_i(\up^{-1}\bb^-)$ is equal to $\ve_i(\bb)-1$ (resp.\ $\ve_i(\bb)$), then $s_is_{i-1}\le w(\bb)$ and $s_is_{i+1}\le w(\bb)$ (resp.\ $s_{i-1}s_i\le w(\bb)$ and $s_is_{i+1}\le w(\bb)$).
Thus the equality for some $0\le l<l_i$ implies $s_is_{i+1}\le w(\bb)$.
Then $\up\chi_{i,i+1}\in \Phi_-$ follows from Corollary \ref{wbCoxeter} (i) and (ii).

If $i\notin \supp(w_\up)$ and $\ld(i)=\ld(i+1)=\mn$, then $(\ld_b(i),\ld_b(i+1))=(\mn, \mn)$ or $(\mn+1,\mn)$.
Note that if the equality holds for some $0\le l<l_i$, then $\ve_i(\up^{-1}\bb^-)=\ve_i(\bb)$.
In the former case, let $j=\min\{j'\mid i+1< j'\le n,\ld(j')=\mn+1\}$.
By $i\notin \supp(w_\up)$, $s_is_{i+1}\cdots s_{j-1}\ld$ is allowed.
Let $\bb'$ be the conjugate of $\bb$ with weight $s_{i+1}\cdots s_{j-1}\ld$.
If the equality holds for some $0\le l<l_i$, then by Lemma \ref{right} (ii) (or Definition \ref{crystaldefi} (iii) if $\ve_i(\bb)=0$) and Remark \ref{remawbve}, we have $\ve_i(\bb')=\ve_i(\bb)+1$ and hence $s_{i+1}s_i\le w(\bb)$.
In the latter case, if the equality holds for some $0\le l<l_i$, then $s_{i+1}s_is_{i-1}\le w(\bb)$ by Lemma \ref{wbve2}.
Thus the equality for some $0\le l<l_i$ implies $s_{i+1}s_i\le w(\bb)$.
Then $\up\chi_{i,i+1}\in \Phi_-$ follows from Corollary \ref{wbCoxeter} (iii) and (iv).

If $i\notin \supp(w_\up)$ and $\ld(i)=\ld(i+1)=\mn+1$, then $(\ld_b(i),\ld_b(i+1))=(\mn+1, \mn+1)$ or $(\mn+1,\mn)$.
Note that if the equality holds for some $0\le l<l_i$, then $\ve_i(\up^{-1}\bb^-)=\ve_i(\bb)$.
In the former case, let $j=\max\{j'\mid 1\le j'< i,\ld(j')=\mn\}$.
By $i\notin \supp(w_\up)$, $s_is_{i-1}\cdots s_j\ld$ is allowed.
Let $\bb'$ be the conjugate of $\bb$ with weight $s_{i-1}\cdots s_j\ld$.
If the equality holds for some $0\le l<l_i$, then by Lemma \ref{right} (ii) (or Definition \ref{crystaldefi} (iii) if $\ve_i(\bb)=0$) and Remark \ref{remawbve}, we have $\ve_i(\bb')=\ve_i(\bb)+1$ and hence $s_{i-1}s_i\le w(\bb)$.
In the latter case, if the equality holds for some $0\le l<l_i$, then $s_{i-1}s_is_{i+1}\le w(\bb)$ by Lemma \ref{wbve2}.
Thus the equality for some $0\le l<l_i$ implies $s_{i-1}s_i\le w(\bb)$.
Then $\up\chi_{i,i+1}\in \Phi_-$ follows from Corollary \ref{wbCoxeter} (ii) and (iv).
This finishes the proof.
\end{proof}

In a similar way, we will prove the following lemma.
\begin{lemm}
\label{inequality2}
\begin{enumerate}[(i)]
\item Assume that $l_{i+1}<l_i$.
We set $\delta=1$ if $i+1\in\supp(w_\up)$ and $\delta=0$  if $i+1\notin\supp(w_\up)$.
Then we have $S_{l}+\delta\le \ve_i(\up^{-1}\bb^-)$ for $l_{i+1}\le l<l_i$.
If $s_is_{i-1}\le w(\bb)$ (resp.\ $s_{i-1}s_i\le w(\bb)$), then the equality for some $l_{i+1}\le l<l_i$ implies that $i\in \supp(w_\up)$ (resp.\ $\ld_b^-(\up(i))=\mn+1$).

\item Assume that $l_{i-1}<l_i$.
We set $\delta=1$ if $i-1\in\supp(w_\up)$ and $\delta=0$  if $i-1\notin\supp(w_\up)$.
Then we have $S_{l}+\delta\le \ve_i(\up^{-1}\bb^-)$ for $l_{i-1}\le l<l_i$.
If $s_is_{i+1}\le w(\bb)$ (resp.\ $s_{i+1}s_i\le w(\bb)$), then the equality for some $l_{i+1}\le l<l_i$ implies that $i\in \supp(w_\up)$ (resp.\ $\ld_b^-(\up(i+1))=\mn$).

\item Assume that $l_{i+1}<l_i$ and $l_{i-1}<l_i$.
We set $\delta=|\{i-1,i+1\}\cap \supp(w_\up)|$.
Then we have $S_{l}+\delta\le \ve_i(\up^{-1}\bb^-)$ for $\max\{l_{i-1},l_{i+1}\}\le l<l_i$.
The equality for some $\max\{l_{i-1},l_{i+1}\}\le l<l_i$ implies that $i\in \supp(w_\up)$.

\end{enumerate}
\end{lemm}
\begin{proof}
We first prove (i).

Assume that $l_{i+1}<l_i$, $s_is_{i-1}\le w(\bb)$ and $\ld_b(i+1)=\mn+1$.
Then $s_is_{i-1}\le w(\bb)$ combined with Lemma \ref{wbve} and Lemma \ref{right} (ii) implies $l_0=l_i$ and hence $S_{l}\le \ve_i(\bb)-1$ for $0\le l<l_i$, where $l_0$ denotes the minimal integer such that $S_{l_0}=\ve_i(\bb)$.
By Lemma \ref{wbve} and $l_{i+1}<l_i$, $i+1\in \supp(w_\up)$ implies $\ve_i(\up^{-1}\bb^-)\geq \ve_i(\bb)$.
Thus the inequality holds.
By $\ld_b(i+1)=\mn+1$, $i+1\in \supp(w_\up)$ implies $i\in \supp(w_\up)$.
If the equality holds for some $l_{i+1}\le l<l_i$ and $i\notin \supp(w_\up)$, then we must have $\ve_i(\up^{-1}\bb^-)=\ve_i(\bb)$ and hence $i+1\in\supp(w_\up)$, which is a contradiction.
Thus the equality implies $i\in \supp(w_\up)$.

Assume that $l_{i+1}<l_i$, $s_{i-1}s_i\le w(\bb)$ and $(\ld_b(i),\ld_b(i+1))=(\mn,\mn+1)$.
Then we have $S_l\le\ve_i(\bb)-1$ for $0\le l<l_i$.
The inequality follows from this, Lemma \ref{wbve} and $l_{i+1}<l_i$.
If the equality holds for some $l_{i+1}\le l<l_i$, then again by Lemma \ref{wbve}, $l_{i+1}<l_i$ and $s_{i-1}s_i\le w(\bb)$, we have $i\in \supp(w_\up)$ and $i-1\notin \supp(w_\up)$.
Thus $\ld_b^-(\up(i))=\mn+1$.

Assume that $l_{i+1}<l_i$, $s_{i-1}s_i\le w(\bb)$ and $(\ld_b(i),\ld_b(i+1))=(\mn+1,\mn+1)$.
The inequality for the case $i+1\notin \supp(w_\up)$ follows from Lemma \ref{inequality}.
If $i+1\in \supp(w_\up)$, then we have $i-1,i\in \supp(w_\up)$.
By Lemma \ref{wbve}, $l_{i+1}<l_i$ and $s_{i-1}s_i\le w(\bb)$, we also have $\ve_i(\up^{-1}\bb^-)\geq \ve_i(\bb)+1$.
Hence the inequality holds.
Note that if $i\in \supp(w_\up)$, then $\ld_b^-(\up(i))=\mn+1$.
So it remains to show that if $i\notin \supp(w_\up)$ and the equality holds for some $l_{i+1}\le l<l_i$, then $i-1\notin \supp(w_\up)$.
By $\ld_b(i+1)=\mn+1$, $i\notin \supp(w_\up)$ implies $i+1\notin \supp(w_\up)$.
So if $i\notin \supp(w_\up)$, then the equality implies $\ve_i(\up^{-1}\bb^-)=\ve_i(\bb)$.
Hence the assertion follows from Lemma \ref{wbve} and $s_{i-1}s_i\le w(\bb)$.
Thus the equality implies $\ld_b^-(\up(i))=\mn+1$.

We next treat the case where $\ld_b(i+1)=\mn$.
For this, we need the following claim.
\begin{claim}
Assume that $l_{i+1}<l_i$ and $\ld_b(i+1)=\mn$.
Then $S_{l}\le \ve_i(\bb)-1$ for $l_{i+1}\le l<l_i$.
\end{claim}
We follow the notation in Corollary \ref{wbve2}.
To check this claim, it suffices to show that if $S_{l}=\ve_i(\bb)$ for some $l_{i+1}\le l$, then $l_i\le l$.
We write $u_i(s_{i+1}\cdots s_{j_1-1}\bb)=u^1\cdots u^{\ld_b(i)+\mn+1}$.
Let $u^{\ell_{i+1}}=-$ be the box added by the action of $s_{i+1}\cdots s_{j_1-1}$.
Let $\ell$ be the maximal integer such that $u^\ell$ is contained in $\bb_{l'}$ with some $l'\le l$.
If $S_{l}=\ve_i(\bb)$ for some $l_{i+1}\le l$, then $\ell_{i+1}\le \ell$, $\ve_i(s_{i+1}\cdots s_{j_1-1}\bb)=\ve_i(\bb)+1$ and the number of $-$ in $(u^1\cdots u^{\ell})_{\red}=-\cdots -$ is $\ve_i(\bb)+1$.
If $\ld_b(i)=\mn$, then $l_i\le l'\le l$ follows immediately from this.
If $\ld_b(i)=\mn+1$, let $u^{\ell_{i-1}}=+$ be the box deleted by the action of $s_{i-1}\cdots s_{j_2}$ on $s_{i+1}\cdots s_{j_1-1}\bb$.
Then the number of $-$ in $u^1\cdots u^{\ell}$ after we delete $u^{\ell_{i-1}}$ (if $\ell_{i-1}\le \ell$) and then ``$+-$'' is $\ve_i(s_{i-1}\cdots s_{j_2}s_{i+1}\cdots s_{j_1-1}\bb)$.
So we have $l_{i+1}\le l'\le l$.
This finishes the proof of Claim 1.

Assume that $l_{i+1}<l_i$ and $\ld_b(i+1)=\mn$.
Then the inequality follows from Lemma \ref{wbve}, Corollary \ref{wbve2} and Claim 1.
By $\ld_b(i+1)=\mn$, $i\in \supp(w_\up)$ implies $i+1\in \supp(w_\up)$.
By $l_{i+1}<l_i$, Lemma \ref{wbve} and Corollary \ref{wbve2}, we have $\ve_i(\up^{-1}\bb^-)\geq \ve_i(\bb)$.
So if the equality holds for some $l_{i+1}\le l<l_i$, then by Claim 1, we must have $i+1\in\supp(w_\up)$ and $\ve_i(\up^{-1}\bb^-)=\ve_i(\bb)$.
If $s_is_{i-1}\le w(\bb)$, then Lemma \ref{wbve} and Corollary \ref{wbve2} imply $\ve_i(s_{i+1}\cdots s_{j_1-1}\bb)=\ve_i(\bb)+1$.
Thus if $s_is_{i-1}\le w(\bb)$ and the equality holds for some $l_{i+1}\le l<l_i$, we have $i\in \supp(w_\up)$.
Also, if $\ld_b(i)=\mn$ and $s_{i-1}s_i\le w(\bb)$, then by Lemma \ref{wbve}, the equality for some $l_{i+1}\le l<l_i$ implies $i\in \supp(w_\up)$ and $i-1\notin \supp(w_\up)$.
Hence $\ld_b^-(\up(i))=\mn+1$.
Note that if $\ld_b(i)=\mn+1$ and $i\in \supp(w_\up)$, then $\ld_b^-(\up(i))=\mn+1$.
Therefore it remains to show that if $\ld_b(i)=\mn+1$, $s_{i-1}s_i\le w(\bb)$, $i\notin \supp(w_\up)$ and the equality holds for some $l_{i+1}\le l<l_i$, then $i-1\notin \supp(w_\up)$.
This follows from Corollary \ref{wbve2}.

Putting things together, we have proved (i).
We can similarly prove (ii) using the following claim.
\begin{claim}
Assume that $l_{i-1}<l_i$ and $\ld_b(i)=\mn+1$.
Then $S_{l}\le \ve_i(\bb)-1$ for $l_{i-1}\le l<l_i$.
\end{claim}
The proof of this claim is also similar to that of Claim 1, so we omit the details.

We next prove (iii).
For this, we need the following claims.
\begin{claim}
Assume that $l_{i-1},l_{i+1}<l_i$.
Then $S_l\le \ve_i(\bb)-1$ for $\max\{l_{i-1},l_{i+1}\}\le l<l_i$.
\end{claim}
This claim is obvious if $(\ld_b(i),\ld_b(i+1))=(\mn,\mn+1)$.
Other cases follow from Claim 1 and Claim 2.

\begin{claim}
Assume that $l_{i-1},l_{i+1}<l_i$ and $(\ld_b(i),\ld_b(i+1))=(\mn+1,\mn)$.
If $\ve_i(s_{i+1}\cdots s_{j_1-1}\bb)=\ve_i(\bb)$ or $\ve_i(s_{i-1}\cdots s_{j_2}\bb)=\ve_i(\bb)$, then $S_l\le \ve_i(\bb)-2$ for $\max\{l_{i-1},l_{i+1}\}\le l<l_i$.
\end{claim}
It follows from Claim 1 and Claim 3 that $S_l\le \ve_i(\bb)-1$ for $\max\{l_{i-1},l_{i+1}\}\le l<l_i$.
We write $u_i(\bb)=u^1\cdots u^{2\mn+1}$.
Let $\ell$ be the maximal integer such that $u^\ell$ is contained in $\bb_{l'}$ with some $l'\le l$.
If $S_l=\ve_i(\bb)-1$ for $\max\{l_{i-1},l_{i+1}\}\le l<l_i$, then $(u^1\cdots u^\ell)_{\red}=-\cdots-$ or $-\cdots -+$.
Here the number of $-$ is $\ve_i(\bb)-1$ or $\ve_i(\bb)$ respectively.
By $\max\{l_{i-1},l_{i+1}\}\le l$ and $l_{i-1},l_{i+1}<l_i$, it follows that $\ve_i(s_{i-1}\cdots s_{j_2}s_{i+1}\cdots s_{j_1-1}\bb)=\ve_i(\bb)+2$ in both cases.
Hence $\ve_i(s_{i+1}\cdots s_{j_1-1}\bb)=\ve_i(\bb)+1$ and $\ve_i(s_{i-1}\cdots s_{j_2}\bb)=\ve_i(\bb)+1$.
This proves the claim.

Assume that $l_{i-1},l_{i+1}<l_i$ and $\delta=0$.
Then the inequality follows from Lemma \ref{inequality}.
If the equality holds for some $\max\{l_{i-1},l_{i+1}\}\le l<l_i$, then by Claim 3, we have $i\in \supp(w_\up)$.

Assume that $l_{i-1},l_{i+1}<l_i$ and $\delta=1$.
Then the inequality follows from (i) and (ii).
If $(\ld_b(i),\ld_b(i+1))= (\mn+1,\mn)$, then by Corollary \ref{wbve2} and Claim 4, the equality never holds.
If the equality holds for some $\max\{l_{i-1},l_{i+1}\}\le l<l_i$ and $(\ld_b(i),\ld_b(i+1))\neq (\mn+1,\mn)$, then by Lemma \ref{wbve} and Claim 3, we must have $i\in \supp(w_\up)$.

Assume that $l_{i-1},l_{i+1}<l_i$ and $\delta=2$.
Then our assertion follows from Claim 3 and Lemma \ref{wbve} (resp.\ Claim 4) if $(\ld_b(i),\ld_b(i+1))\neq (\mn+1,\mn)$ (resp.\ $(\ld_b(i),\ld_b(i+1))= (\mn+1,\mn)$).
This finishes the proof of (iii).
\end{proof}

For $S_{l_i}$, we have the following lemma with the same notation as in Corollary \ref{wbve2}.
\begin{lemm}
\label{equality}
If $(\ld_b(i),\ld_b(i+1))=(\mn+1,\mn)$, $s_{i-1}s_i\le w(\bb)$, $s_{i+1}s_i\le w(\bb)$ and $\ve_i(s_{i-1}\cdots s_{j_2}s_{i+1}\cdots s_{j_1-1}\bb)=\ve_i(\bb)+1$, then $S_{l_i}=\ve_i(\bb)-1$. 
Otherwise, we have $S_{l_i}=\ve_i(\bb)$.
\end{lemm}
\begin{proof}
This is obvious if $(\ld_b(i),\ld_b(i+1))=(\mn,\mn+1)$.
If $(\ld_b(i),\ld_b(i+1))=(\mn,\mn)$ and $\ve_i(s_{i+1}\cdots s_{j_1-1}\bb)=\ve_i(\bb)$, then $S_{l_i}=\ve_i(\bb)$ follows from Lemma \ref{right} (ii).
If $(\ld_b(i),\ld_b(i+1))=(\mn,\mn)$ and $\ve_i(s_{i+1}\cdots s_{j_1-1}\bb)=\ve_i(\bb)+1$, then $\la\chi_{i+1,i},\wt(\bb'_1)+\cdots +\wt(\bb'_{l_i})\ra=\ve_i(\bb)+1$, where $\FE(s_{i+1}\cdots s_{j_1-1}\bb)=\bb'_1\otimes\cdots\otimes \bb'_d$.
By Remark \ref{remawbve}, the box added by the action of $s_{i+1}\cdots s_{j_1-1}$ is contained in one of $\bb'_1,\ldots,\bb'_{l_i}$.
This implies $S_{l_i}=\ve_i(\bb)$.
The proof for the case $(\ld_b(i),\ld_b(i+1))=(\mn+1,\mn+1)$ is similar.
By Lemma \ref{wbve} and Corollary \ref{wbve2}, the proof for the case $(\ld_b(i),\ld_b(i+1))=(\mn+1,\mn)$ is also similar.
\end{proof}

Set $T_l=\la\chi_{i,i+1},\wt(\bb_{l_i+1})+\cdots+\wt(\bb_{l})\ra$ for $l_i<l$.
We also set $T_{l_i}=0$.
We will also need the following inequality.
\begin{lemm}
\label{inequality3}
For $l_i\le l$, we have $T_l\geq 0$.
If $s_is_{i-1}\le w(\bb)$ (resp.\ $s_is_{i+1}\le w(\bb)$) and the equality holds for some $l_{i-1}\le l$ (resp.\ $l_{i+1}\le l$), then $\ld_b(i)=\mn$ (resp.\ $\ld_b(i+1)=\mn+1$).
Similarly, if $s_is_{i-1}\le w(\bb)$, $s_is_{i+1}\le w(\bb)$ and $T_l=1$ for some $\max\{l_{i-1},l_{i+1}\}\le l$, then $\ld_b(i)=\mn$ or $\ld_b(i+1)=\mn+1$.
\end{lemm}
\begin{proof}
Let $\ld$ be an allowed cocharacter with $i\notin\supp(w)$ and $\la\chi_{i,i+1},\ld\ra=-1$, where $w$ is the partial Coxeter element associated with $\ld$.
Let $\bb'$ be be the conjugate of $\bb$ with $\wt(\bb')=\ld$.
Set $\FE(\bb')=\bb'_1\otimes \cdots \otimes \bb'_d$.
Then the action of $s_i$ on $\bb'$ changes a box in $\bb'_{l_i}$.
Since $\la\chi_{i,i+1},\wt(\bb'_{l_i+1})+\cdots+\wt(\bb'_{l})\ra$ is the difference of the number of $+$ and $-$ in $u_i(\bb')$ which are contained in $\bb'_{l_i+1},\ldots, \bb'_{l}$, we have $T_l\geq \la\chi_{i,i+1},\wt(\bb'_{l_i+1})+\cdots+\wt(\bb'_{l})\ra\geq 0$.

Assume that $s_is_{i+1}\le w(\bb)$ and the equality holds for $l_{i+1}\le l$.
If $l=l_i$, then $l_i=l_{i+1}$ and hence $\ld_b(i+1)=\mn+1$.
Assume moreover that $l_i<l$.
To show $\ld_b(i+1)=\mn+1$, we argue by contradiction.
If $\ld_b(i+1)=\mn$, i.e., $s_is_{i+1}\le w'_{\max}$, then $s_is_{i+1}\le w(\bb)$ implies $l_i<l_{i+1}$.
So if $\ld_b(i+1)=\mn$ and the equality holds for $l_{i+1}\le l$, then $\la\chi_{i,i+1},\wt(\bb'_{l_i+1})+\cdots+\wt(\bb'_{l})\ra\le -1$.
This implies $\la\chi_{i+1,i},\wt(\bb'_1)+\cdots+\wt(\bb'_{l})\ra\geq \ve_i(\bb')+1$, which is a contradiction.
The rest of the statement follows in the same way.
The proof is finished.
\end{proof}

\begin{rema}
In \S\ref{computation} and \S\ref{inequalities}, $l_i$ denotes an integer such that $i\in\supp(w_{l_i})$.
However, in the proof of Proposition \ref{upsilon general} or Proposition \ref{wbtensor}, $l_h$ denotes an integer such that $j_h\in \supp(w_{l_h})$.
We hope our notation will not cause confusions.
\end{rema}

\subsection{Proof of $\up_{\xi_l(\bb,\up)}=\up w_1^{-1}\cdots w_{l-1}^{-1}$}
\label{upsilonsubsection}
Fix $\bb\in \B_{\mu}(\ld_b)$ and $\up\in \Upsilon(\bb)$.
Set $\up_l=\up w_1^{-1}w_2^{-1}\cdots w_{l-1}^{-1}$ for $\up\in \Upsilon(\bb)$ and $1\le l\le d$.
We write $\xi_l$ for $\xi_l(\bb, \up)$.
The goal of this section is to prove $\up_{\xi_l}=\up_l$.

Fix a reduced expression $s_{j_1}s_{j_2}\cdots s_{j_{n-1}}$ of $w(\bb)$ such that
\begin{align*}
&w_1^{-1}=s_{j_1}s_{j_2}\cdots s_{j_{\ell(w_1)}},\\
&w_2^{-1}=s_{j_{\ell(w_1)+1}}s_{j_{\ell(w_1)+2}}\cdots s_{j_{\ell(w_1)+\ell(w_2)}},\\
&\ \vdots \\
&w_{d}^{-1}=s_{j_{\ell(w_1)+\cdots +\ell(w_{d-1})+1}}s_{j_{\ell(w_1)+\cdots +\ell(w_{d-1})+2}}\cdots s_{j_{\ell(w_1)+\cdots+\ell(w_d)}}.
\end{align*}
Define $1\le l_h\le d$ by $j_h\in \supp(w_{l_h})$.
Then $l_1\le l_2\le\cdots \le l_{n-1}$.
\begin{lemm}
\label{upsilon claim}
For each $1\le h\le n-1$ such that $l_h\le d-1$, we have
\begin{align*}
&j_h\in \supp(w_\up)\Leftrightarrow \la \chi_{j_h, j_h+1},s_{j_{h-1}}\cdots s_{j_2}s_{j_1}\up^{-1}\xi_{l_h+1}\ra=-1, \\
&j_h\notin \supp(w_\up)\Leftrightarrow \la \chi_{j_h, j_h+1},s_{j_{h-1}}\cdots s_{j_2}s_{j_1}\up^{-1}\xi_{l_h+1}\ra=0.
\end{align*}
\end{lemm}
\begin{proof}
We first prove the case for $h=1$.
Note that $s_{j_1}s_{j_1-1}\le w(\bb)$ and $s_{j_1}s_{j_1+1}\le w(\bb)$.
So by Lemma \ref{wbve}, $j_1\in \supp(w_\up)$ (resp.\ $j_1\notin \supp(w_\up)$) if and only if $\ve_{j_1}(\up^{-1}\bb^-)=\ve_{j_1}(\bb)-1$ (resp.\ $\ve_{j_1}(\up^{-1}\bb^-)=\ve_{j_1}(\bb)$).
By Lemma \ref{equality}, we have $\la\chi_{j_1+1,j_1},\wt(\bb_1)+\cdots+\wt(\bb_{l_1})\ra=\ve_{j_1}(\bb)$.
This and $\up^{-1}\xi_{l_1+1}=\xi(\up^{-1}\bb^-)+\wt(\bb_1)+\cdots+\wt(\bb_{l_1})$ imply that we have $j_1\in \supp(w_\up)$ (resp.\ $j_1\notin \supp(w_\up)$) if and only if $\la \chi_{j_1, j_1+1}, \up^{-1}\xi_{l_1+1}\ra=-1$ (resp.\ $\la \chi_{j_1, j_1+1}, \up^{-1}\xi_{l_1+1}\ra=0$).

Assume that our claim is true for $1,2,\ldots, h-1$ with $h\geq 2$.
If $j_h-1,j_h+1\notin \{j_1,j_2,\ldots,j_{h-1}\}$, then 
\begin{align*}
\la\chi_{j_h,j_h+1},s_{j_{h-1}}\cdots s_{j_2}s_{j_1}\up^{-1}\xi_{l_h+1}\ra=&\ve_{j_h}(\up^{-1}\bb^-)+\la\chi_{j_h,j_h+1},\wt(\bb_1)+\cdots+\wt(\bb_{l_h})\ra
\end{align*} 
and the statement follows in the same way of the case for $h=1$.
If $j_{h'}=j_h+1$ for some $1\le h'\le h-1$ and $j_h-1\notin  \{j_1,j_2,\ldots,j_{h-1}\}$, then
\begin{align*}
\la\chi_{j_h,j_h+1},s_{j_{h-1}}\cdots s_{j_2}s_{j_1}\up^{-1}\xi_{l_h+1}\ra=&\ve_{j_h}(\up^{-1}\bb^-)+\la\chi_{j_h,j_h+1},\wt(\bb_1)+\cdots+\wt(\bb_{l_h})\ra\\
&+\la\chi_{j_{h'},j_{h'}+1},s_{j_{h'-1}}\cdots s_{j_2}s_{j_1}\up^{-1}\xi_{l_{h'}+1}\ra.
\end{align*}
Then the assertion follows from Lemma \ref{wbve}, Corollary \ref{wbve2}, Lemma \ref{equality} and the induction hypothesis.
The proof for the case $j_{h'}=j_h-1$ for some $1\le h'\le h-1$ and $j_h+1\notin  \{j_1,j_2,\ldots,j_{h-1}\}$ is similar.
If $\{j_{h'},j_{h''}\}=\{j_h-1,j_h+1\}$ for $1\le h'<h''\le h-1$, then 
\begin{align*}
\la\chi_{j_h,j_h+1},s_{j_{h-1}}\cdots s_{j_2}s_{j_1}\up^{-1}\xi_{l_h+1}\ra=&\ve_{j_h}(\up^{-1}\bb^-)+\la\chi_{j_h,j_h+1},\wt(\bb_1)+\cdots+\wt(\bb_{l_h})\ra\\
&+\la\chi_{j_{h'},j_{h'}+1},s_{j_{h'-1}}\cdots s_{j_2}s_{j_1}\up^{-1}\xi_{l_{h'}+1}\ra\\
&+\la\chi_{j_{h''},j_{h''}+1},s_{j_{h''-1}}\cdots s_{j_2}s_{j_1}\up^{-1}\xi_{l_{h''}+1}\ra.
\end{align*}
By Corollary \ref{suppwcoro} (iv), the case where $s_{j_h-1}s_{j_h}\le w(\bb)$, $s_{j_h+1}s_{j_h}\le w(\bb)$ and $j_h-1,j_h+1\notin \supp(w_\up)$ does not occur.
Then the assertion follows from this, Lemma \ref{wbve}, Corollary \ref{wbve2}, Lemma \ref{equality} and the induction hypothesis.
Thus the statement is true for $h$.
By induction, this finishes the proof.
\end{proof}

\begin{prop}
\label{upsilon general}
We have $U_{\xi_l(\bb, \up)}=\up_l U\up_l^{-1}$, i.e., $\up_{\xi_l}=\up_l$ for any $1\le l\le d$.
\end{prop}

\begin{proof}
We will prove $\up_{\xi_{l+1}}=\up_{l+1}$ for $0\le l\le d-1$.
For this, we have to check that $\la \chi_{i,i+1},\up_{l+1}^{-1}\xi_{l+1}\ra\geq 0$ for any $1\le i\le n-1$ and that if $\la \chi_{i,i+1},\up_{l+1}^{-1}\xi_{l+1}\ra=0$, then $\up_{l+1}\chi_{i,i+1}\in \Phi_-$.
Set $\supp_l=\supp(w_1)\cup\cdots \cup\supp(w_l)$ (and $\supp_0=\emptyset$).

Let $i\notin \supp_l$.
Note that
$\up^{-1}_{l+1}\xi_{l+1}=w_{l}\cdots w_1\xi(\up^{-1}\bb^-)+w_{l}\cdots w_1\wt(\bb_1)+\cdots +w_{l}\wt(\bb_{l})$.
So if $i-1,i+1\notin \supp_l$, then the assertion follows from Lemma \ref{inequality}.
If $i+1\in\supp_l$ and $i-1\notin \supp_l$, then
\begin{align*}
\la\chi_{i,i+1},\up_{l+1}^{-1}\xi_{l+1}\ra&=\ve_i(\up^{-1}\bb^-)+\la\chi_{i,i+1},\wt(\bb_1)+\cdots+\wt(\bb_{l})\ra\\
&\hspace{0.6cm}+\la\chi_{j_{h},j_{h}+1},s_{j_{h-1}}\cdots s_{j_2}s_{j_1}\up^{-1}\xi_{l_{h}+1}\ra,
\end{align*}
where $j_{h}=i+1$.
Thus the assertion follows from  Lemma \ref{suppw}, Corollary \ref{suppwcoro} (ii), Lemma \ref{inequality2} (i) and Lemma \ref{upsilon claim}.
The proof for the case where $i-1\in\supp_l$ and $i+1\notin \supp_l$ is similar.
If $i-1,i+1\in\supp_l$, then \begin{align*}
\la\chi_{i,i+1},\up_{l+1}^{-1}\xi_{l+1}\ra=&\ve_i(\up^{-1}\bb^-)+\la\chi_{i,i+1},\wt(\bb_1)+\cdots+\wt(\bb_{l})\ra\\
&+\la\chi_{j_{h},j_{h}+1},s_{j_{h-1}}\cdots s_{j_2}s_{j_1}\up^{-1}\xi_{l_{h}+1}\ra\\
&+\la\chi_{j_{h'},j_{h'}+1},s_{j_{h'-1}}\cdots s_{j_2}s_{j_1}\up^{-1}\xi_{l_{h'}+1}\ra,
\end{align*}
where $\{j_{h},j_{h'}\}=\{i-1,i+1\}$ with $h<h'$.
Thus the assertion follows from Lemma \ref{suppw}, Lemma \ref{inequality2} (iii) and Lemma \ref{upsilon claim}.
Therefore our assertion is true for $i\notin \supp_l$.

Let $i\in \supp_l$.
Let $h$ such that $j_{h}=i$.
We set $\supp_{l,h}=\supp_l\setminus \{j_1,\ldots, j_{h}\}$.
If $j_{h}-1,j_{h}+1\notin \supp_{l,h}$, then 
\begin{align*}
&\la \chi_{j_{h},j_{h}+1},\up_{l+1}^{-1}\xi_{l+1}\ra\\
=&\la \chi_{j_{h}+1,j_{h}},s_{j_{h-1}}\cdots s_{j_2}s_{j_1}\up^{-1}\xi_{l_{h}+1}\ra+\la \chi_{j_{h},j_{h}+1},\wt(\bb_{l_{h}+1})+\cdots+\wt(\bb_l)).
\end{align*}
By Lemma \ref{inequality3}, $\la \chi_{j_{h},j_{h}+1},\wt(\bb_{l_{h}+1})+\cdots+\wt(\bb_l))\geq 0$.
Then the assertion follows from Lemma \ref{upsilon claim}.
If $j_{h}-1\notin \supp_{l,h}$ and $j_{h}+1\in \supp_{l,h}$, then
\begin{align*}
&\la \chi_{j_{h},j_{h}+1},\up_{l+1}^{-1}\xi_{l+1}\ra\\
=&\la \chi_{j_{h},j_{h}+1},\up_{l+1}^{-1}\xi_{l_{h'}+1}\ra+\la \chi_{j_{h},j_{h}+1},\wt(\bb_{l_{h'}+1})+\cdots+\wt(\bb_{l})\ra\\
=&\la\chi_{j_{h}+1,j_{h}},s_{j_{h-1}}\cdots s_{j_2}s_{j_1}\up^{-1}\xi_{l_{h'}+1}\ra+\la\chi_{j_{h'},j_{h'}+1},s_{j_{h'-1}}\cdots s_{j_2}s_{j_1}\up^{-1}\xi_{l_{h'}+1}\ra\\
&+\la \chi_{j_{h},j_{h}+1},\wt(\bb_{l_{h'}+1})+\cdots+\wt(\bb_{l})\ra\\
=&\la\chi_{j_{h}+1,j_{h}},s_{j_{h-1}}\cdots s_{j_2}s_{j_1}\up^{-1}\xi_{l_{h}+1}\ra+\la\chi_{j_{h},j_{h}+1}, \wt(\bb_{l_{h}+1})+\cdots+\wt(\bb_{l})\ra\\
&+\la\chi_{j_{h'},j_{h'}+1},s_{j_{h'-1}}\cdots s_{j_2}s_{j_1}\up^{-1}\xi_{l_{h'}+1}\ra,
\end{align*}
where $j_{h'}=j_{h}+1$.
By Lemma \ref{inequality3}, $\la \chi_{j_{h},j_{h}+1},\wt(\bb_{l_{h}+1})+\cdots+\wt(\bb_{l})\ra=0$ implies $\ld_b(j_{h}+1)=\mn+1$.
Hence $j_{h}+1\in \supp(w_\up)$ implies $j_{h}\in \supp(w_\up)$.
Thus $\la \chi_{j_{h},j_{h}+1},\up_{l_h+1}^{-1}\xi_{l_h+1}\ra\geq 0$ by Lemma \ref{upsilon claim}.
The equality holds if and only if one of the following case occurs:
\begin{itemize}
\item $j_{h}\notin \supp(w_\up)$, $j_{h'}\in \supp(w_\up)$ and $\la \chi_{j_{h},j_{h}+1},\wt(\bb_{l_{h}+1})+\cdots+\wt(\bb_{l})\ra=1$,
\item $j_{h}\notin \supp(w_\up)$, $j_{h'}\notin \supp(w_\up)$ and $\la \chi_{j_{h},j_{h}+1},\wt(\bb_{l_{h}+1})+\cdots+\wt(\bb_{l})\ra=0$,
\item $j_{h}\in \supp(w_\up)$, $j_{h'}\in \supp(w_\up)$ and $\la \chi_{j_{h},j_{h}+1},\wt(\bb_{l_{h}+1})+\cdots+\wt(\bb_{l})\ra=0$.
\end{itemize}
In the first case, $\up_{l+1}\chi_{j_{h},j_{h}+1}=\up s_{j_1}\cdots s_{j_{h}-1}\chi_{j_{h}+1,j_{h}}+\up s_{j_1}\cdots s_{j_{h'}-1}\chi_{j_{h'},j_{h'}+1}\in \Phi_-$ by Lemma \ref{upsilon claim}.
In the last two cases, we have $\ld_b^-(\up(j_{h}+1))=\ld_b(j_{h}+1)=\mn+1$ and hence $\up(j_{h}+1)>n-m_0$.
If $j_{h}-1\in \{j_1,\ldots,j_{h-1}\}$, $\up_{l+1}\chi_{j_{h},j_{h}+1}=\up w(\bb)\chi_{j_{h},j_{h}+1}=\tau^m \up\chi_{j_{h},j_{h}+1}\in \Phi_-$ by Corollary \ref{wbCoxeter} (ii) (if $\up(j_{h})>n-m_0$).
If $j_{h}-1\notin \{j_1,\ldots,j_{h-1}\}$, then $\up_{l+1}\chi_{j_{h},j_{h}+1}\in \Phi_-$ follows from Corollary \ref{suppwcoro} (i).
Thus $\up_{l+1}\chi_{j_{h},j_{h}+1}\in \Phi_-$ holds in every case.
The proof for the case where $j_{h}+1\notin \supp_{l,h}$ and $j_{h}-1\in \supp_{l,h}$ is similar.
If $j_{h}-1,j_{h}+1\in \supp_{l,h}$, then by Lemma \ref{inequality3}, $\la \chi_{j_{h},j_{h}+1},\wt(\bb_{l_{h}+1})+\cdots+\wt(\bb_{l})\ra=0$ (resp.\ $1$) implies $(\ld_b(j_{h}),\ld_b(j_{h}+1))=(\mn,\mn+1)$ (resp.\ $\ld_b(j_{h})=\mn$ or $\ld_b(j_{h}+1)=\mn+1$).
Thus the inequality follows similarly as above.
Using Corollary \ref{wbCoxeter} (i), we can also check that the equality implies $\up_{l+1}\chi_{j_{h},j_{h}+1}\in \Phi_-$ in the same way as above.
Therefore our assertion is true for $i\in \supp_l$.
This completes the proof.
\end{proof}

\subsection{End of The Proof}
\label{end}
In this subsection, we finish the proof of Theorem \ref{constructionthm}.
\begin{lemm}
\label{tf}
Let $\bb'\in \B_\mu$ and let $1\le i\le n-1$.
If $\ve_i(\bb')>0$, then let $l$ be the positive integer such that 
$$\FE(\te_i\bb')=\bb'_1\otimes\cdots \otimes \te_i\bb'_l\otimes\cdots\otimes \bb'_d,$$
where $\FE(\bb')=\bb'_1\otimes\cdots\otimes \bb'_d$.
If $\ve_i(\bb')=0$, set $l=0$.
Then the action of $\tf_i^{\phi_i(\bb')}$ on $\bb'$ does not affect the boxes in $\bb'_1,\ldots, \bb'_l$ and the following equality holds:
$$\wt(\tf_i^{\phi_i(\bb')}\bb')=\sum_{j=1}^{l}\wt(\bb'_j)+s_i(\sum_{j=l+1}^{d}\wt(\bb'_j)).$$
\end{lemm}
\begin{proof}
We naturally identify $\bb'_j$ and $\wt(\bb'_j)$.
We need to check that the $i,i+1$-th entries in both sides are equal (because other entries are clearly equal).
We write $u_i(\bb')=u^1\cdots u^{\wt(\bb')(i)+\wt(\bb')(i+1)}$.
Let $u^\ell=-$ be the box in $\bb'_l$ (which is changed to $+$ by the action of $\te_i$).
Then $(u^{\ell+1}\cdots u^{\wt(\bb')(i)+\wt(\bb')(i+1)})_{\red}=+\cdots +$ and the number of $+$ here is equal to $\phi_i(\bb')$.
Note that $\tf_i^{\phi_i(\bb')}$ changes all $+$ in this diagram to $-$.
Note also that $\tf_i^{\phi_i(\bb')}$ does not affect the boxes in $\bb'_1,\ldots, \bb'_l$ and ``$+-$'' in $u_i(\bb')$, which we neglect in $u_i(\bb')_{\red}$.
On the other hand, the action of $s_i$ on $\bb'_{l+1},\ldots,\bb'_d$ changes $+$ to $-$ and $-$ to $+$.
It is easy to see the total number of $+(=\fbox{$i$})$ or $-(=\fbox{$i+1$})$ in both sides are equal.
Hence the equality holds.
\end{proof}

\begin{prop}
\label{wbtensor}
Let $\bb\in \B_{\mu}(\ld_b)$ and $\up\in \Upsilon(\bb)$.
Then $$w(\bb)(\xi(\up^{-1}\bb^-)+\up^{-1}\ld_b^-)=\xi(\up^{-1}\bb^-)+\sum_{1\le j\le d}w_1^{-1}\cdots w_{j-1}^{-1}\wt(\bb_j).$$
\end{prop}
\begin{proof}
Fix a reduced expression $s_{j_1}s_{j_2}\cdots s_{j_{n-1}}$ of $w(\bb)$.
We define $$\Phi_{j_h}(\bb,\up)\coloneqq
\begin{cases}
\phi_{j_h}(\up^{-1}\bb^-) & (s_{j_h-1}s_{j_h+1}s_{j_h}\le w(\bb))\\
\phi_{j_h}(\up^{-1}\bb^-)+\Phi_{j_h-1}(\bb,\up) & (s_{j_h+1}s_{j_h}s_{j_h-1}\le w(\bb))\\
\phi_{j_h}(\up^{-1}\bb^-)+\Phi_{j_h+1}(\bb,\up) & (s_{j_h-1}s_{j_h}s_{j_h+1}\le w(\bb))\\
\phi_{j_h}(\up^{-1}\bb^-)+\Phi_{j_h-1}(\bb,\up)+\Phi_{j_h+1}(\bb,\up) & (s_{j_h}s_{j_h-1}s_{j_h+1}\le w(\bb))
\end{cases}
$$
inductively from $h=n-1$ to $1$, setting $\Phi_{0}(\bb,\up)=\Phi_{n}(\bb,\up)=0$.
In particular, we have $\Phi_{j_{n-1}}(\bb,\up)=\phi_{j_{n-1}}(\up^{-1}\bb^-)$.
Write $\Phi_{j_h}$ for $\Phi_{j_h}(\bb,\up)$.
First, we prove that 
$$w(\bb)(\bb_{\xi(\up^{-1}\bb^-)}\otimes \up^{-1}\bb^-)=\bb_{\xi(\up^{-1}\bb^-)}\otimes\tf_{j_1}^{\Phi_{j_1}}\tf_{j_2}^{\Phi_{j_2}}\cdots\tf_{j_{n-1}}^{\Phi_{j_{n-1}}}(\up^{-1}\bb^-)$$
by induction (see Example \ref{highest} for $\bb_{\xi(\up^{-1}\bb^-)}$).
Since $$\phi_{j_{n-1}}(\bb_{\xi(\up^{-1}\bb^-)})=\la\chi_{j_{n-1},j_{n-1}+1},\xi(\up^{-1}\bb^-)\ra=\ve_{j_{n-1}}(\up^{-1}\bb^-)$$
and 
\begin{align*}\la\chi_{j_{n-1},j_{n-1}+1},\xi(\up^{-1}\bb^-)+\up^{-1}\ld_b^-\ra&=\ve_{j_{n-1}}(\up^{-1}\bb^-)+\la\chi_{j_{n-1},j_{n-1}+1},\up^{-1}\ld_b^-\ra\\
&=\phi_{j_{n-1}}(\up^{-1}\bb^-)\geq 0,
\end{align*}
we have 
\begin{align*}
s_{j_{n-1}}(\bb_{\xi(\up^{-1}\bb^-)}\otimes \up^{-1}\bb^-)&=\bb_{\xi(\up^{-1}\bb^-)}\otimes\tf_{j_{n-1}}^{\phi_{j_{n-1}}(\up^{-1}\bb^-)}(\up^{-1}\bb^-)\\
&=\bb_{\xi(\up^{-1}\bb^-)}\otimes\tf_{j_{n-1}}^{\Phi_{j_{n-1}}}(\up^{-1}\bb^-),
\end{align*}
cf.\  Definition \ref{tensor}.
Let $l_{h}$ be an integer such that $j_h\in \supp(w_{l_h})$.
We write $\FE(\up^{-1}\bb^-)=\bb'_1\otimes \cdots \otimes \bb'_{d}$.
Let $u=+$ be the leftmost $+$ in $u_{j_{n-1}}(\up^{-1}\bb^-)_{\red}$ if it exists, i.e., $\Phi_{j_{n-1}}\neq 0$.
Note that the action of $s_{j_{n-1}-1}$ and $s_{j_{n-1}+1}$ along the way of computing $\up^{-1}\bb^-$ from $\bb$ does not increase $\phi_i$.
So if $j_{n-1}\in \supp(w_\up)$, then $u$ is the box in $\bb'_{l_{n-1}},\ldots,\bb'_{d}$.
Equivalently, if $u$ is the box in $\bb'_1,\ldots,\bb'_{l_{n-1}-1}$, then $j_{n-1}\notin \supp(w_\up)$.
Moreover, $j_{n-1}\in \supp(w_{l_{n-1}})$ implies $\la \chi_{j_{n-1},j_{n-1}+1},\ld_b\ra=\la \chi_{j_{n-1},j_{n-1}+1},\up^{-1}\ld_b^-\ra=1$ in this case.
By $s_{j_{n-1}-1}s_{j_{n-1}+1}s_{j_{n-1}}\le w(\bb)$, this contradicts to Corollary \ref{suppwcoro} (iv).
Thus $\tf_{j_{n-1}}^{\Phi_{j_{n-1}}}$ does not change the boxes in $\bb'_1,\ldots,\bb'_{l_{n-1}-1}$.
In fact, if $j_{n-1}\in \supp(w_\up)$, then by $s_{j_{n-1}-1}s_{j_{n-1}+1}s_{j_{n-1}}\le w(\bb)$, $u$ must be the box in $\bb'_{l_{n-1}}$ (which is changed from $-$ along the computation $\up^{-1}\bb^-$ from $\bb$).
If $u$ is the box in $\bb'_{l_{n-1}}$ and $j_{n-1}\notin \supp(w_\up)$, then we must have $\la \chi_{j_{n-1},j_{n-1}+1},\ld_b\ra=\la \chi_{j_{n-1},j_{n-1}+1},\up^{-1}\ld_b^-\ra=1$, which contradicts to Corollary \ref{suppwcoro} (iv).
Therefore $\tf_{j_{n-1}}^{\Phi_{j_{n-1}}}$ changes the box in $\bb'_{l_{n-1}}$ if and only if $j_{n-1}\in \supp(w_\up)$.

Assume that 
$$s_{j_{h+1}}s_{j_{h+2}}\cdots s_{j_{n-1}}(\bb_{\xi(\up^{-1}\bb^-)}\otimes \up^{-1}\bb^-)=\bb_{\xi(\up^{-1}\bb^-)}\otimes\tf_{j_{h+1}}^{\Phi_{j_{h+1}}}\tf_{j_{h+2}}^{\Phi_{j_{h+2}}}\cdots\tf_{j_{n-1}}^{\Phi_{j_{n-1}}}(\up^{-1}\bb^-)$$
for some $h<n-1$.
We further assume that for any $h'>h$, $\tf_{j_{h'}}^{\Phi_{j_{h'}}}$ does not change the boxes in $\bb_1,\bb_2,\ldots,\bb_{l_{h'}-1}$, and $\tf_{j_{h'}}^{\Phi_{j_{h'}}}$ changes the box in $\bb_{l_{h'}}$ if and only if $j_{h'}\in \supp(w_\up)$.
It easily follows from Definition \ref{crystaldefi} (i) that 
$$\la\chi_{j_h,j_h+1}, \wt(\bb_{\xi(\up^{-1}\bb^-)}\otimes\tf_{j_{h+1}}^{\Phi_{j_{h+1}}}\tf_{j_{h+2}}^{\Phi_{j_{h+2}}}\cdots\tf_{j_{n-1}}^{\Phi_{j_{n-1}}}(\up^{-1}\bb^-))\ra=\Phi_{j_h}.$$
Moreover, we have
$$\ve_{j_h}(\tf_{j_{h+1}}^{\Phi_{j_{h+1}}}\tf_{j_{h+2}}^{\Phi_{j_{h+2}}}\cdots\tf_{j_{n-1}}^{\Phi_{j_{n-1}}}(\up^{-1}\bb^-))=\ve_{j_h}(\up^{-1}\bb^-).$$
This is obvious if $s_{j_h-1}s_{j_h+1}s_{j_h}\le w(\bb)$.
If $j_h\notin \supp(w_\up)$ and $s_{j_h}s_{j_h-1}\le w(\bb)$ (resp.\ $s_{j_h}s_{j_h+1}\le w(\bb)$), then by the induction hypothesis, $\tf_{j_h-1}^{\Phi_{j_h-1}}$ (resp.\ $\tf_{j_h+1}^{\Phi_{j_h+1}}$) does not change the box in $\bb_1,\ldots,\bb_{l_h}$.
Indeed, if $j_h-1\in \supp(w_\up)$ (resp.\ $j_h+1\in \supp(w_\up)$), then $l_h<l_{h-1}$ (resp.\ $l_h<l_{h+1}$).
Note that the action of $\tf_{j_h-1}^{\Phi_{j_h-1}}$ (resp.\ $\tf_{j_h+1}^{\Phi_{j_h+1}}$) does not increase $\ve_{j_h}$ and hence $\ve_{j_h}(\tf_{j_{h+1}}^{\Phi_{j_{h+1}}}\tf_{j_{h+2}}^{\Phi_{j_{h+2}}}\cdots\tf_{j_{n-1}}^{\Phi_{j_{n-1}}}(\up^{-1}\bb^-))\le \ve_{j_h}(\up^{-1}\bb^-)$.
Note also that if $j_h\notin \supp(w_\up)$, then by Lemma \ref{right} (ii), there exists $0\le l\le l_h$ such that $\la\chi_{j_h+1,j_h},\wt(\bb'_1)+\cdots +\wt(\bb'_{l})\ra=\ve_{j_h}(\up^{-1}\bb^-)$.
Hence this equality holds if $j_h\notin \supp(w_\up)$.
If $j_h\in \supp(w_\up)$, then there exists $0\le l< l_h$ such that $\la\chi_{j_h+1,j_h},\wt(\bb'_1)+\cdots +\wt(\bb'_{l})\ra=\ve_{j_h}(\up^{-1}\bb^-)$ except if $s_{j_h}s_{j_h-1}\le w(\bb)$, $s_{j_h}s_{j_h+1}\le w(\bb)$ and $\la\chi_{j_h, j_h+1},\ld_b\ra= -1$.
In fact, this exceptional case does not occur by Corollary \ref{suppwcoro} (iii).
By the induction hypothesis, $\tf_{j_h-1}^{\Phi_{j_h-1}}$ (resp.\ $\tf_{j_h+1}^{\Phi_{j_h+1}}$) does not change the box in $\bb_1,\ldots,\bb_{l_h-1}$.
So this equality also holds in this case.

Thus, by $\phi_{j_h}(\bb_{\xi(\up^{-1}\bb^-)})=\ve_{j_h}(\up^{-1}\bb^-)$ and the induction hypothesis, we have
$$s_{j_h}s_{j_{h+1}}\cdots s_{j_{n-1}}(\bb_{\xi(\up^{-1}\bb^-)}\otimes \up^{-1}\bb^-)=\bb_{\xi(\up^{-1}\bb^-)}\otimes\tf_{j_h}^{\Phi_{j_h}}\tf_{j_{h+1}}^{\Phi_{j_{h+1}}}\cdots\tf_{j_{n-1}}^{\Phi_{j_{n-1}}}(\up^{-1}\bb^-).$$
Moreover $\tf_{j_{h}}^{\Phi_{j_{h}}}$ does not change the boxes in $\bb'_1,\ldots,\bb'_{l_{h}-1}$, and $\tf_{j_{h}}^{\Phi_{j_{h}}}$ changes the box in $\bb'_{l_{h}}$ if and only if $j_{h}\in \supp(w_\up)$.
Indeed, if $s_{j_h-1}s_{j_h+1}s_{j_h}\le w(\bb)$, then this follows similarly as above.
Assume that $s_{j_h}s_{j_h-1}\le w(\bb)$ or $s_{j_h}s_{j_h+1}\le w(\bb)$.
Note that if $j_{h'}=j_h-1$ (resp.\ $j_h+1$) for some $h<h'\le n-1$, then $l_{h}\le l_{h'}$ and the action of $\tf_{j_{h'}}^{\Phi_{j_{h'}}}$ adds $+$ (resp.\ deletes $-$) in $u_{j_h}(\up^{-1}\bb^-)$. 
Let $u=+$ be the leftmost $+$ in $u_{j_h}(\tf_{j_h}^{\Phi_{j_h}}\tf_{j_{h+1}}^{\Phi_{j_{h+1}}}\cdots\tf_{j_{n-1}}^{\Phi_{j_{n-1}}}(\up^{-1}\bb^-))$.
If $u$ is the box in $\bb'_1,\ldots,\bb'_{l_{h}-1}$, then $j_h\notin \supp(w_\up)$.
By Lemma \ref{wbve} and our assumption, this contradicts to $j_h\in \supp(w_{l_h})$.
Thus $\tf_{j_{h}}^{\Phi_{j_{h}}}$ does not change the boxes in $\bb'_1,\ldots,\bb'_{l_{h}-1}$.
If $j_h\in \supp(w_\up)$, then $u$ is the box in $\bb'_{l_h}$.
Indeed, if $l_h=l_{h'}$ and $j_{h'}\in \supp(w_\up)$, then $\tf_{j_{h'}}^{\Phi_{j_{h'}}}$ changes the box in $\bb'_{l_{h}}$.
If $u$ is the box in $\bb'_{l_h}$ and $j_h\notin \supp(w_\up)$, then $j_h\in \supp(w_{l_h})$ implies $\la\chi_{j_h,j_h+1},\ld_b\ra=1$ and $j_h-1,j_h+1\in \supp(w_{l_h})$.
This contradicts to our assumption that $s_{j_h}s_{j_h-1}\le w(\bb)$ or $s_{j_h}s_{j_h+1}\le w(\bb)$.
Thus $\tf_{j_{h}}^{\Phi_{j_{h}}}$ changes the box in $\bb'_{l_{h}}$ if and only if $j_{h}\in \supp(w_\up)$.
By induction, this finishes the computation of $w(\bb)(\bb_{\xi(\up^{-1}\bb^-)}\otimes \up^{-1}\bb^-)$.

Since
\begin{align*}
w(\bb)(\xi(\up^{-1}\bb^-)+\up^{-1}\ld_b^-)&=\wt(w(\bb)(\bb_{\xi(\up^{-1}\bb^-)}\otimes \up^{-1}\bb^-))\\
&=\xi(\up^{-1}\bb^-)+\wt(\tf_{j_1}^{\Phi_{j_1}}\tf_{j_2}^{\Phi_{j_2}}\cdots\tf_{j_{n-1}}^{\Phi_{j_{n-1}}}(\up^{-1}\bb^-)),
\end{align*}
it remains to show that
$$\wt(\tf_{j_1}^{\Phi_{j_1}}\tf_{j_2}^{\Phi_{j_2}}\cdots\tf_{j_{n-1}}^{\Phi_{j_{n-1}}}(\up^{-1}\bb^-))=\sum_{1\le j\le d}w_1^{-1}\cdots w_{j-1}^{-1}\wt(\bb_j).$$
In the above discussion, we have proved that
$$\phi_{j_h}(\tf_{j_{h+1}}^{\Phi_{j_{h+1}}}\tf_{j_{h+2}}^{\Phi_{j_{h+2}}}\cdots\tf_{j_{n-1}}^{\Phi_{j_{n-1}}}(\up^{-1}\bb^-))=\Phi_{j_h}$$
and that $\tf_{j_{h}}^{\Phi_{j_{h}}}$ changes the box in $\bb_{l_{h}}$ if and only if $j_{h}\in \supp(w_\up)$.
Note that $\Phi_{j_1}\geq \cdots \geq \Phi_{j_{n-1}}$.
Thus we can easily check this equality by applying Lemma \ref{tf} repeatedly.
The proof is finished.
\end{proof}

\begin{proof}[Proof of Theorem \ref{constructionthm}]
We first show 
\begin{align*}
b\xi_1(\bb,\up)=\up\xi(\up^{-1}\bb^-)+\sum_{1\le j\le d}\up w_1^{-1}\cdots w_{j-1}^{-1}\wt(\bb_j). \tag{$\ast$}
\end{align*}
Note that $b=\tau^m \vp^{\ld_b^+}$ as an element of $\tW$, where $\ld_b^+$ is the dominant conjugate of $\ld_b$.
So
\begin{align*}
(\ast)&\Leftrightarrow \tau^m\up \xi(\up^{-1}\bb^-)+\ld_b^+=\up\xi(\up^{-1}\bb^-)+\sum_{1\le j\le d}\up w_1^{-1}\cdots w_{j-1}^{-1}\wt(\bb_j)\\
&\Leftrightarrow \up^{-1}\tau^m\up(\xi(\up^{-1}\bb^-)+\up^{-1}\ld_b^-)=\xi(\up^{-1}\bb^-)+\sum_{1\le j\le d}w_1^{-1}\cdots w_{j-1}^{-1}\wt(\bb_j).
\end{align*}
Since $\up^{-1}\tau^m\up=w(\bb)$, the last equality follows from Proposition \ref{wbtensor}.
This shows $(\ast)$.
By $(\ast)$ and Proposition \ref{upsilon general}, we have $\xi_\bl(\bb, \up)^\flat=\FE(\bb)$.
By Theorem \ref{minuscule bijection}, this implies $\xi_{\bl}(\bb,\up)\in \cA_{\mu_\bl, b_\bl}^\tp,(\Gamma^{G^d})^{-1}(\FE(\bb))=[\xi_{\bl}(\bb,\up)]$ and $\xi_{\bl}(\bb, \up)\sim \xi_{\bl}(\bb, \up')$ for any $\up, \up'\in \Upsilon(\bb)$.
Since $\up_{\xi_1(\bb, \up)}=\up$ and $\up_{\xi_1(\bb, \up')}=\up'$, $\up\neq \up'$ implies $\xi_{\bl}(\bb, \up)\neq \xi_{\bl}(\bb, \up')$.
The proof is finished.
\end{proof}

\bibliographystyle{myamsplain}
\bibliography{reference}
\end{document}